\documentclass[reqno]{amsart}
\usepackage{amsmath}
\usepackage{amsthm}
\usepackage{amsfonts}
\usepackage{amssymb}
\usepackage{graphicx}
\usepackage{mathrsfs}
\usepackage{cite}
\usepackage{textcomp}

\title[Integer points in a simplex] 
{Integer points in a simplex 
	and related Diophantine problems: 
Hardy-Littlewood asymptotics in higher dimensions}
\author[M.M. SKRIGANOV]{M.M. SKRIGANOV}

\address{St. Petersburg Department of the Steklov Mathematical Institute 
of the Russian Academy of Sciences, 
27, Fontanka, St.Petersburg, 191023, Russia}

\email{maksim88138813@mail.ru}



\numberwithin{equation}{section}

\newtheorem{theorem}{Theorem}[section]
\newtheorem{lemma}{Lemma}[section]
\newtheorem{proposition}{Proposition}[section]

\theoremstyle{remark}
\newtheorem{remark}{Remark}[section]
\theoremstyle{remark}
\newtheorem{definition}{Definition}[section]

\def\bfa{\mathbf{a}}
\def\bfb{\mathbf{b}}
\def\bfc{\mathbf{c}}

\def\bfe{\mathbf{e}}

\def\bfl{\mathbf{l}}
\def\bfm{\mathbf{m}}
\def\bfn{\mathbf{n}}

\def\bfu{\mathbf{u}}
\def\bfv{\mathbf{v}}
\def\bfw{\mathbf{w}}
\def\bfx{\mathbf{x}}
\def\bfy{\mathbf{y}}
\def\bfz{\mathbf{z}}

\def\dd{\mathrm{d}}

\def\Cc{\mathbb{C}}

\def\Ff{\mathbb{F}}

\def\Qq{\mathbb{Q}}
\def\Rr{\mathbb{R}}

\def\Zz{\mathbb{Z}}

\def\BBB{\mathcal{B}}
\def\CCC{\mathcal{C}}

\def\EEE{\mathcal{E}}

\def\III{\mathcal{I}}

\def\MMM{\mathcal{M}}
\def\NNN{\mathcal{N}}

\def\RRR{\mathcal{R}}
\def\SSS{\mathcal{S}}

\def\UUU{\mathcal{U}}

\def\XXX{\mathcal{X}}

\def\BBBB{\mathscr{B}}

\def\LLLL{\mathscr{L}}

\def\SSSS{\mathscr{S}}

\renewcommand{\le}{\leqslant}
\renewcommand{\ge}{\geqslant}
\renewcommand{\Re}{\mathop{\mathrm{Re}}\nolimits}

\numberwithin{equation}{section}

\numberwithin{equation}{section}
\theoremstyle{plain}

\newcommand{\bp}{\begin{proof}}
\newcommand{\ep}{\end{proof}}
\newcommand{\bl}{\begin{lemma}}
\newcommand{\el}{\end{lemma}}
\newcommand{\bt}{\begin{theorem}}
\newcommand{\et}{\end{theorem}}
\newcommand{\bd}{\begin{definition}}
\newcommand{\ed}{\end{definition}}
\newcommand{\ba}{\begin{arrow}}
\newcommand{\ea}{\end{arrow}}

\begin{document}

\begin{abstract}
The aim of the present work is to study the asymptotic behavior of the numbers of 
integer points $N^{\mp}(t;\bfw)$  in  d-dimensional 
open and closed right-angled simplices
\begin{align*}
	&N^{-} (t;\bfw)=\#\{\bfx =(x_1,\dots ,x_d)\in\Zz^d: x_j>0, j\in 
	[d],\,\, \bfw\centerdot\bfx <t\}\, ,
	\\
	&N^{+} (t;\bfw)=\#\{\bfx =(x_1,\dots ,x_d)\in\Zz^d: x_j\ge0, j\in 
	[d],\,\, \bfw\centerdot\bfx \le t\}
	\end{align*} 
as $t\to\infty$. 
Here $\bfw\centerdot\bfx = w_1x_1+\cdots +w_d x_d$
and $\bfw =(w_1,\dots w_d)$ are positive weights such that all quotients
$\theta_{j,l}=\frac{w_l}{w_j},\,l\ne j,$ are \textit{irrational}.

For $d=2$, the problem was investigated by Hardy and 
Littlewood in the early 1920s.
For an arbitrary dimension $d$, we prove the asymptotic formulas
\begin{align*}
 N^{\mp}(t;\bfw)=\frac{1}{d!\, w_1\cdots w_d } 
\BBB_d^{\star}(t\mp\bfw\centerdot\bfe_{1/2};\bfw) 
+R^{\mp}(t;\bfw),
\end{align*} 
where $\bfe_{1/2}=(1/2,\dots,1/2)$ and $\BBB^{\star}_d(t;\bfw)$
is a
\textit{multiple Bernoulli polynomial} of degree 
$d$
\begin{equation*}
	\BBB^{\star}_d(t;\bfw)=\sum\nolimits_{0\le 
		n\le\lfloor\frac{d}{2}\rfloor}\,\frac{k!}{(2n)!(k-2n)!}\, 
	\BBB^{\star}_{2n}(\bfw)\,\, t^{k-2n}\, ,
\end{equation*}
where
\begin{align*}
	\BBB^{\star}_{2n}(\bfw)
	=\sum_{n_1+\dots 
		+n_d=n}\,\frac{(2n)!}{(2n_1)!\dots (2n_d)!}\, 
	B_{2n_1}(1/2)\dots B_{2n_d}(1/2)\, w_1^{2n_1}\dots w_d^{2n_d}.
\end{align*}
and $B_k(1/2)$ are the values of the Bernoulli polynomials $B_k(u)$ at $u=1/2$. 

The error terms $R^{\mp}(t;\bfw)$ turn out to be intimately related to  
\textit{multiplicative Diophantine approximation} to 
$\theta_{j,l}=\frac{w_l}{w_j},\,l\ne j,$.
In particular, we prove the following bounds: 

\textit{(i)} $R^{\mp}(t;\bfw)=O_{\epsilon}((\log t)^{d+\epsilon})$ 
for any $\epsilon >0$, for almost all unit vectors 
$\widehat{\bfw}$ which are normal to the hyperplane $\{\bfx:\bfw\centerdot\bfx=0\}$,

\textit{(ii)} $R^{\mp}(t;\bfw)=O_{\epsilon}(t^{\epsilon})$ 
for any $\epsilon >0$
for positive algebraic numbers $w_1,\dots, w_d$ linearly independent over 
the field of rationals $\Qq$.

\end{abstract}

\enlargethispage{5\baselineskip}

\maketitle

\thispagestyle{empty}

\section*{Contents}

\noindent{} {1. Introduction} 


\noindent{} {2. Multiple Bernoulli polynomials}

\noindent{} {3. Multiplicative Diophantine approximation}
 

\noindent{} {4. Main Results}


\noindent{} {5. Simplices and their Fourier transform}

\noindent{} {6. Fourier expansions for the error terms}

\noindent{} {7. Smoothed Fourier expansions. Proof of the Main Theorem}


References

\medskip

\section{Introduction}
\label{sec1}

Counting integer points in large compact domains is a classic number 
theory problem. 
In this paper, we consider the problem of counting the number of integer 
points in open and closed $d$-dimensional right-angled simplices:
\begin{align*}
N^{\mp}(t;\bfw)= \#\,\{\,\Delta^{\mp} 
	(t;\bfw)\,\cap\,\Zz^d\,\}, 
\end{align*}
where
\begin{align*}
&\Delta^{-} (t;\bfw)=\{\bfx:x_j>0,\, j\in 
[d], 
\,\,\bfw\centerdot\bfx<t\},	
	\\
&\Delta^{+} (t;\bfw)=\{\bfx:x_j\ge 0,\, 
j\in [d], 
\,\,\bfw\centerdot\bfx\le t\}.
\end{align*}
Here $\bfx\centerdot\bfw =\sum\nolimits_1^d x_j w_j$,\, $\bfw=(w_1\dots 
w_d)\in\Rr^d_{>0}$ and $[d]=\{1, \dots, d\}$. 
It is clear that
$N^{\mp}(t;\bfw)$ are piecewise constant,
monotone non-decreasing functions of $t\in\Rr_{>0}$; $N^{-}(t;\bfw)$ is 
left-continuous and $N^{+}(t;\bfw)$ is right-continuous.

A specific feature of this 
problem is the extreme sensitivity of 
the corresponding asymptotics 
to arithmetic 
characteristics of the weights $\bfw=(w_1,\dots,w_d)$.

The study of $N^{\mp}(t;\bfw)$  for integer $\bf 
w\in\Zz^d_{>0}$ is 
a special case of Ehrhart theory for counting integer 
point  for convex polytopes in $\Rr^d$ whose vertices 
have integer coordinates (lattice polytopes). 
For such polytopes and integer $t$, $N^{\mp}(t;\bfw)$  are  
polynomials of degree $d$  whose coefficients expressed 
in terms of Fourier-Dedekind sums. 
Various aspects of this theory are closely related to number theory, algebraic 
geometry, and combinatorics.
	
In the book \cite{4}, the reader will find a detailed description of 
the Ehrhart theory for general lattice polytopes; see also
\cite{10, 16*, 18} for the case of lattice simplices. 

In this paper, we consider the opposite situation, where 
 $w_1,\dots,w_d$ are positive weights such that all quotients
$\theta_{j,l}=\frac{w_l}{w_j},\,l\ne j,$ are irrational. 
The number-theoretic 
problems that arise in this case  belong to 
the theory of simultaneous Diophantine approximation.

In the early 1920s, Hardy and Littlewood \cite{13, 14} (reprinted in \cite{14a})
considered the number of integer points 
$N^{\mp}(t;\bfw)$ in the  right-angled triangle 
$\Delta^{\mp}(t;\bfw)$ with $\bfw=(w_1,w_2)$ 
and an irrational $\theta=\frac{w_2}{w_1}$. 
Let us put
\begin{equation}
	N^{\mp}(t;\bfw)=\frac{t^2}{2w_1w_2}\mp\frac{t}{2w_1}\mp
	\frac{t}{2w_2}+R^{\mp}(t;\bfw),
	\label{eq1.3}
\end{equation}

The  results obtained by Hardy and Littlewood, together with the 
subsequent results of other authors, can be summarized as follows.

\textit{(i)} 
The following bound is contained in \cite[Eq. (4.152)]{13}
\begin{equation*}
	R^{\mp}(t;\bfw)=O\big(\,\sum\nolimits_{j=1}^{c\log t}\, a_j\,\big)\, ,
\end{equation*}
where $a_j,\, j\in\Zz_{>0} ,$ are partial quotients in the 
continued fraction of $\theta$ and $c>0$ is an absolute constant.
Therefore,
\begin{equation*}
	R^{\mp}(t;\bfw)=O_{\theta}(\,\log t\,)\, ,
\end{equation*}
if the partial quotients of $\theta$ are bounded -- for 
example, if $\theta$ is a quadratic irrationality (see \cite[Thm A3]{13}).

Khinchin \cite{15}, using his results on the metric theory of continued 
fractions, proved that 
\begin{equation}
	R^{\mp}(t;\bfw)=O_{\theta,\epsilon}(\,\log t\, (\log\log t 
	)^{1+\epsilon}\,)\, ,
	\label{eq1.6}
\end{equation}
for any $\epsilon >0$ for almost all $\theta\in\Rr_{>0}$.

\textit{(ii)}
The following result is contained in 
\cite[Thm 2]{14}. Suppose that the inequality 
\begin{equation}
	m^{1+\kappa}\,|\sin\pi m \theta\,|\approx m^{1+\kappa}\,\langle  
	m\theta\rangle >c 
	\label{eq1.7}
\end{equation}
holds for all $m\in\Zz_{>0}$ with some constants $\kappa> 0$ and 
$c>0$. Then
\begin{equation*}
	R^{\mp}(t;\bfw)=O_{\epsilon}\big(\, t^{\frac{\kappa}{1+\kappa} 
	+\epsilon}\,\big)
\end{equation*}
for any $\epsilon >0$. Here
 $\langle x \rangle$ denotes the distance from $x\in\Rr$ to
the nearest integer, and the $\approx$ sign denotes the two side inequalities. 

By Roth's theorem on Diophantine approximation (unknown in the 
1920s),  inequality \eqref{eq1.7}
 holds for any 
$\kappa>0$ for each real algebraic irrationality $\theta$,
see \cite[Chapter V, Theorem 2A]{19}.
Therefore,  for such $\theta$,
\begin{equation*}
	R^{\mp}(t;\bfw)=O_{\theta,\epsilon}(\, t^{\epsilon}\,)
\end{equation*}
for any $\epsilon >0$.
\\

The aim of the present work is to extend the Hardy--Littlewood asymptotics 
\eqref{eq1.3} to higher dimensions. 
For an arbitrary dimension $d$, we  prove the asymptotic formulas
\begin{align}
	N^{\mp}(t;\bfw)=\LLLL^{\mp}_d(t;\bfw)+R^{\mp}(t;\bfw),
	\label{eq1.19I}
\end{align}
where the leading terms are defined as
\begin{align*}
	\LLLL^{\mp}_d(t;\bfw)=\frac{1}{d!\, w_1\cdots w_d } 
	\BBB_d^{\star}(t\mp\bfw\centerdot\bfe_{1/2};\bfw) 
\end{align*}
where $\bfe_{1/2}=(1/2,\dots,1/2)$ and $\BBB^{\star}_d(t;\bfw)$
is the $d$-th polynomial in the following sequence of
\textit{multiple Bernoulli polynomials}  
\begin{equation}
	\BBB^{\star}_k(t;\bfw)=\sum\nolimits_{0\le 
		n\le\lfloor\frac{k}{2}\rfloor}\,\frac{k!}{(2n)!(k-2n)!}\, 
	\BBB^{\star}_{2n}(\bfw)\,\, t^{k-2n}\,,\quad k\in\Zz_{\ge 0}\,,
	\label{eq1.19A}
\end{equation}
where
\begin{align*}
	\BBB^{\star}_{2n}(\bfw)
	=\sum_{n_1+\dots 
		+n_d=n}\,\frac{(2n)!}{(2n_1)!\dots (2n_d)!}\, 
	B_{2n_1}(1/2)\dots B_{2n_d}(1/2)\, w_1^{2n_1}\dots w_d^{2n_d}.
\end{align*}
and $B_k(1/2)$ are the values of the standard Bernoulli polynomials $B_k(u)$ at 
$u=1/2$.

The error terms $R^{\mp}(t;\bfw)$ are intimately related to  
\textit{multiplicative Diophantine approximation} of 
$\theta_{j,l}=\frac{w_l}{w_j},\,l\ne j$.
We will prove the following bounds, see Theorem~4.3.

\textit{(i)}
For almost all unit vectors 
$\widehat{\bfw}$ normal to the hyperplane $\{\bfx:\bfw\centerdot\bfx=0\}$, 
\begin{equation}
	R^{\mp}(t;\bfw)=O_{\widehat\bfw,\epsilon}\,\big(\,(\log\,t)^{d+\epsilon}\,\big)
	\label{eq1.3aI}
\end{equation}
for any $\epsilon >0$. 

\textit{(ii)} 
Suppose that the inequalities 
\begin{equation}
	m^{1+\kappa}\,\prod\nolimits_{l\in [d],\,l\ne j}\,\, \langle 
	\theta_{j,l}\,\, m\rangle >c, 
	\quad j\in [d]\,,
	\label{eq1.18A}
\end{equation}
hold  for all $m\in \Zz_{>0}$ with some constants $\kappa>0$ and
$c>0$ simultaneously for all $j\in [d]$.
Then
\begin{equation}
	R^{\mp}(t;\bfw)=O_{\bfw}\,\big(\, t^{\frac{\kappa}{1+\kappa} 
		(d-1)+\epsilon}\,\big) 
	\label{eq3bI}
\end{equation}
for any $\epsilon>0$.

By Schmidt's theorem on Diophantine approximation, the inequalities \eqref{eq1.18A}
hold for any 
$\kappa>0$ if the weights $\bfw=(w_1,\dots,w_d)$ are algebraic numbers linearly 
independent 
over $\Qq$,
see \cite[Chapter V, Theorem 2A]{19}.
Hence,  for such $\bfw$,
\begin{equation}
	R^{\mp}(t;\bfw)=O_{\bfw,\epsilon}\big(\, t^{\,\epsilon} \,\big)
	\label{eq3cI}
\end{equation}
for any $\epsilon >0$.

Note that the exponent 
$\frac{\kappa}{1+\kappa} (d-1)$ in \eqref{eq3bI} is the best possible and, in 
general, 
cannot be improved. The corresponding examples are given at the end of 
Section~4.
The bound \eqref{eq3bI} can be replaced by 
$O_{\bfw}\,\big(\, t^{\frac{\kappa}{1+\kappa} 
	(d-1)}\,(\log t)^{d-1}\,\big) $
but we do not consider such a minor improvement because 	 
it would require a much more complex technique, see \cite{22a}.


It is instructive to take a closer look at the leading terms.
It follows from \eqref{eq1.19A} that the leading terms satisfy the symmetry
relation
\begin{equation}
	\LLLL_d^{(-)}(-t;\bfw)=(-1)^d\,\LLLL_d^{(+)}(t;\bfw)\,. 
	\label{eq1.34}
\end{equation}

The leading terms for the first few dimensions can be found explicitly.
For $d=2$, we find
\begin{align*}
	\LLLL_2^{(\mp)}(t;w_1,w_2)=
	\frac{1}{2w_1w_2} \,\Bigl(\,&\Bigl(t\mp\frac{w_1+w_2}{2}\Bigr)^2 
	-\frac{w_1^2+w_2^2}{12} 
	\,\Bigr) 
	\\
	&= \frac{t^2}{2w_1w_2}\mp\frac{t}{2w_1}\mp\frac{t}{2w_2}+
	\frac{w_1^2+w_2^2+3w_1w_2}{12w_1w_2}	
\end{align*}
in agreement  with \eqref{eq1.3}. For $d=3$ and 4, we find
\begin{align*}
	\LLLL_3^{(\mp)}&(t;w_1,w_2,w_3)
	\notag
	\\
	=&\frac{1}{6w_1w_2w_3} \,\left(\, \Bigl(t\mp\frac{w_1+w_2+w_3}{2}\Bigr)^3 -
	\frac{w_1^2+w_2^2+w_3^2}{4} \,\,\Bigl(t\mp\frac{w_1+w_2+w_3}{2}\Bigr)
	\right)  	
\end{align*}
and
\begin{align*}
	&\LLLL_4^{(\mp)}(t;w_1,\dots,w_4)
	\notag
	\\
	&=\frac{1}{24w_1w_2w_3w_4} \,\left(\, 
	\Bigl(t\mp\frac{w_1+\cdots +w_4}{2}\bigr)^4 +
	c_2(\bfw)\,\Bigl(t\mp\frac{w_1+\cdots +w_4}{2}\Bigr)^2 + c_0(\bfw)
	\,\right), 
\end{align*}
where
\begin{align*}
	c_2(\bfw)&=-\frac{1}{2}\left(\,w_1^2 +\dots + w_4^2\,\right),
	\\
	c_0(\bfw)&=\frac{7}{240}
	\left(\, w_1^4+\cdots +w_4^4 \,\right) + 
	\frac{1}{24} \left(\,w_1^2 w_2^2 +\cdots 
	+w_3^2w_4^2\,\right).
\end{align*}

Expanding the parentheses in such formulas, we obtain the leading terms written in 
the form
\begin{align*}
\LLLL^{-}_d(t;\bfw)&=\frac{1}{d!\, w_1\cdots w_d }\,\,  
	\BBB_d(t;\bfw)\,,
	\\
	\LLLL^{+}_d(t;\bfw)&=\frac{1}{d!\, w_1\cdots w_d }\,\,
	\,\BBB_d(\,t+\bfw\centerdot\bfe_1;\bfw)\,,	
\end{align*}
where $\bfe_1=(1,\dots,1)$ and $\BBB_d(t;\bfw)$ 
is the $d$-th polynomial in the following sequence of
\textit{multiple Bernoulli polynomials}
\begin{equation}
	\BBB_k(t; \bfw)
	=\sum\nolimits_{n=0}^d\,\frac{k!}{n!(k-n)!}\,\, 
	\BBB_n(\bfw)\,\, t^{k-n}\,,\quad k\in\Zz_{\ge 0}\,,
	\label{eq2.10I}
\end{equation}
where 
\begin{equation*}
	\BBB_n(\bfw)=\BBB_n(\bfw,{\bf{0}})=\sum_{n_1+\dots 
		+n_d=n}\,\frac{n!}{n_1!\dots n_d!}\, 
	B_{n_1}\dots B_{n_d}\,\, w_1^{n_1}\dots w_d^{n_d}\,,
\end{equation*}
where $B_n$ are the standard Bernoulli numbers.

The polynomials $\BBB^{*}_k(t; \bfw)$ and $\BBB_k(t; \bfw)$ are related by
\begin{equation} 
\BBB_k(t\centerdot\bfu;\,\bfw)=
\BBB_k^{\star}(\,t+\bfw\centerdot(\bfu-\bfe_{1/2});\,\bfw\,)\,
\label{eq2.22I}
\end{equation}
(see Lemma~2.1 below).
Hence—and this circumstance should be emphasized—the leading terms in the 
asymptotic formulas under consideration can be written in very different ways, 
although at first glance such transformations of the formulas may not seem obvious.
We discuss these issues in Section~2.

It is worth noting that the asymptotics for $N^{\mp}(t;\bfw)$  can be easily 
extended to more general 
polytopes that are disjoint unions of  right-angled simplices of different 
dimensions.  
A typical example is the three-dimensional octahedron or, more generally, the 
d-dimensional cross-polytope:
\begin{align*}
	C^{\Delta} (t;\bfw)=\{\,\bfx:\, \sum\nolimits_{j\in [d]}\,w_j\, |x_j|<\,t\,\}\,.
\end{align*}
The corresponding asymptotic formulas for the number of integer points in such 
polytopes would be of interest, but we will not overload 
the paper with their detailed consideration.

Our approach to proving the asymptotic formulas \eqref{eq1.19I} is as follows. 
In addition to $N^{\mp}(t;\bfw)$, we  consider the number of 
integer points in shifted simplices:
\begin{equation*}
	N(t;\bfw,\bfu)=\#\,\{\,\Delta^{-} (t;\bfw,  
	\bfu)\,\cap\,\Zz^d\,\}, 
\end{equation*}
where $\Delta^{-} (t;\bfw,\bfu)=\Delta^{-} (t;\bfw)+\bfu$ is the simplex 
$\Delta^{-} (t;\bfw)$ shifted by $\bfu\in\Rr^d$: 
Clearly that $N(t;\bfw,\bfu)$ is a periodic function: 
$N(t;\bfw,\bfu 
+\bfz)=N(t;\bfw,\bfu)$ for $\bfz\in\Zz^d$. 

In Sections~5 and 6, we  calculate the Fourier expansions of $N(t;\bfw,\bfu)$.
Of course, these multidimensional Fourier series are not convergent, and a 
smoothing technique is used in Section 7 to handle them.
As a result, we reduce the bounds for the error terms in the asymptotic 
formulas to bounds for special lattice sums of the form 
\begin{equation}
	S(\theta, T)=
	\sum_{m=1}^T\,\,
	\frac{1}{m\,\langle\,\theta_1\,m\, 
		\rangle\dots\langle\,\theta_{d-1}\,m\,\rangle}\,,
	\label{eq1.34I}
\end{equation}
with  $\theta=(\theta_1,\dots,\theta_{d-1})\in\Rr^{d-1}$ irrational.
For $d=2$, such sums were studied by Hardy and Littlewood \cite{13, 14}
in connection with the asymptotics \eqref{eq1.3}; see also \cite{2A, 3A}.
For $d\ge3$, the sums \eqref{eq1.34I} will be estimated in Proposition~3.1 and 
Proposition~3.2 in terms of multiple Diophantine approximations. 
For example, in Proposition~3.2, we prove that if the inequalities \eqref{eq1.18A}
are satisfied with $\kappa>0$, then
\begin{equation*}
	S(\theta, T)=O_{\theta}\,\big(\, T^{\kappa} 
	\,(\log T)^{d-2}\,\big)\, , 
\end{equation*}
It is worth noting that the proof of this estimate for $d\ge 3$ turns out to 
be not much more difficult than for $d=2$.

Our main results on the asymptotics of $N(t;\bfw,\bfu)$ are given in Theorems~4.1 
and 4.2. As a corollary of these results (for special shifts $\bfu$), we 
obtain  the asymptotics for $N^{\mp}(t;\bfw)$ in Theorem~4.3.
\\

The reader will probably benefit from the following few additional comments.

Hardy and Littlewood \cite{13, 14} 
established a connection between the asymptotics \eqref{eq1.3} 
and the double Barnes zeta function. Spencer \cite{24} (reprinted in \cite{24a}) 
extended their result to 
higher dimensions and proved the following: for almost all $\bfw\in\Rr^d_{>0}$,
\begin{equation}
	N^{-}(t;\bfw)=(-1)^d\,\zeta (0, t, \bfw)+O_{\epsilon,\bfw}((\log 
	t)^{d+\epsilon}), 
	\label{eq1.37a}
\end{equation}
for any $\epsilon >0$. 
Here $\zeta (0, t, \bfw)$ is the value of the Barnes zeta function at  
$s=0$. Recall 
that the Barnes zeta function is a multidimensional generalization of  
Riemann--Hurwitz zeta functions, it is defined by
\begin{equation*}
	\zeta (s, t, \bfw)=\sum\nolimits_{\bfm\in\Zz^d_{\ge 0}}\, 
	(t+\bfw\centerdot\bfm)^{-s}
\end{equation*}
for $t>0,\,\bfw\in\Rr^d_{>0},\,\Re s>d$ and continued meromorphically to $\Cc$ 
with only simple 
poles at $s=1, \dots, d$. An extensive literature is devoted to these 
issues, see, for example; \cite{1, 3, 5*, 11, 11*}.

In \cite{24}, it was  noted that $\zeta (-k, t, \bfw)$ for $k\in\Zz_{\ge0}$ is a 
polynomial of degree $d+k$ whose coefficients are symmetric functions of 
$w_1,\dots,w_d$. The explicit values  are as follows:
\begin{equation*}
	\zeta (-k, t, \bfw)=\frac{(-1)^d k!}{w_1\dots w_d\,(k+d)!}\, 
	\BBB_{k+d}(t;\bfw),\quad k\in\Zz_{\ge 0},
\end{equation*}
where $\BBB_n(t;\bfw)$ are the multiple Bernoulli polynomials \eqref{eq2.10I},
see the references cited above.

For $k=0$, taking \eqref{eq2.22I} into account,  we obtain
\begin{equation*}
	\zeta (0, t, \bfw)=\frac{(-1)^d }{w_1\dots w_d\,d!}\, 
	\BBB_d(t;\bfw)=\frac{(-1)^d }{w_1\dots w_d\,d!}\, 
	\BBB^{\star}_d(t-\bfw\centerdot\bfe_{1/2};\bfw)
\end{equation*}
Hence, the leading terms in \eqref{eq1.37a}
and \eqref{eq1.19I} coincide.

Note that the metric bound for the error in \eqref{eq1.37a} is 
given in terms of the $d$-dimensional Lebesgue measure on $\Rr^d_{>0}$.
In other words,  \eqref{eq1.37a} holds for almost all 
$\widehat\bfw\in S^{d-1}_{>0}$ along sequences $t\to\infty$. 
Our bound \eqref{eq1.3aI} can be regarded as a refinement of \eqref{eq1.37a}. 

It is likely that the exponent $d+\epsilon$ in \eqref{eq1.37a} 
can be  replaced by $d-1+\epsilon$. 
This is the case in  two dimensions in view of Khinchin's bound \eqref{eq1.6}. It 
is known that such an improvement has 
been achieved for $d\ge3$ in similar problems in  discrepancy theory and in 
 lattice point counting for polytopes \cite{4*, 21, 22, 23}. 
A more detailed 
discussion of these issues is outside the scope of this paper.

In fact,  metric bounds such as  \eqref{eq1.3aI} 
are given in the paper to emphasize that sub--polynomial estimates for the errors 
are typical of the problem under consideration.
The proof of individual bounds of the type \eqref{eq3bI} and \eqref{eq3cI} is 
the most substantial 
part of our work.
\\

Asymptotic formulas for the number of integer points in the $d$-dimensional 
simplices $\Delta^{-}(t;\bfw)$ and cross-polytopes $C^{\Delta} (t;\bfw)$ with 
algebraic $(w_1,\dots,w_d)$ linearly independent over $\Qq$ were 
considered earlier by Borda \cite{5}. Unfortunately,
the leading terms in the asymptotics  were described in 
\cite[Definitions 2 and 5]{5} in a rather complicated way, 
and their relationship with multiple Bernoulli 
polynomials was not  clarified.
The error terms were estimated in \cite[Thm. 8]{5} as
$O(t^{\frac{(d-1)(d-2)}{2d-3}+\epsilon})$, which is very far from 
optimal.
However, a useful formula for the Fourier transform of the indicator function of a 
simplex,  given in  \cite[Eq.(18)]{5}, is used below in Lemma~5.1 to 
simplify some of our calculations.
\\


It should be also noted that the asymptotic formulas for $N^{\mp}(t;\bfw)$ 
have interesting 
applications to  problems in mathematical physics.
Of course, we do not have the opportunity to discuss  these issues here;
the reader may refer to 
\cite{8, 8a} 
and the references therein.
\\


\section{Multiple Bernoulli polynomials}
The polynomials of interest  belong to a 
broad class of so-called 
Appell polynomials.
Let $a_0=1, a_1, a_2, \dots$ be an arbitrary sequence of 
complex numbers. The polynomials $A_k(t),\, k\in\Zz_{\ge 0},$ are 
defined by the 
following equality in the ring of formal power series
\begin{equation*}
	\left(\sum\nolimits_{n\ge 0} 
	a_n\frac{s^n}{n!}\right)\,e^{ts}=
	\left(\sum\nolimits_{n\ge 0} 
	a_n\frac{s^n}{n!}\right)\,\left(\sum\nolimits_{m\ge 0}\frac{t^m 
		s^m}{m!}\right)
	=\sum\nolimits_{k\ge 0} A_k(t)\frac{s^k}{k!}.
\end{equation*}
A direct calculation gives 
\begin{equation*}
	A_k(t)=\sum\nolimits_{n=0}^k\,\frac{k!}{n!(k-n)!}\,\, a_n\,\, t^{k-n}.
\end{equation*}
These polynomials satisfy the conditions: 
$\deg A_k(t)=k,\, A_k(0)=a_k,\, A_0(t)=1$, 
$\frac{\dd}{\dd t}A_k(t)=kA_{k-1}(t)$ for $k\in\Zz_{>0}$; moreover, any sequence 
of polynomials satisfying these conditions can be obtained by the above 
construction. 
Numerous sequences of classical polynomials belong to the Appell class; see 
\cite{16**}, \cite[Chap. 2]{12**}.

For example, the standard 
Bernoulli polynomials $B_k(t)$ are defined by the generating function
\begin{equation}
	\frac{s}{e^s -1}\,\, e^{st} = \sum\nolimits_{k\ge 0}\, B_k(t)\, 
	\frac{s^k}{k!}\,.
	\label{eq2.3}
\end{equation}
We have
\begin{equation*}
	B_k(t)=\sum\nolimits_{n=0}^k\,\frac{k!}{n!(k-n)!}\, B_n\, t^{k-n},
\end{equation*}
where $B_k=B_k(0)$ are Bernoulli numbers defined by the expansion
\begin{equation*}
	\frac{s}{e^s -1} = \sum\nolimits_{k\ge 0}\, B_k\, \frac{s^k}{k!},
\end{equation*} 
$B_0=1, B_1=1/2, B_2=1/6, B_3=0, B_4=-1/30$, and  
$B_k=0$ for odd $k\ge 3$. 
All necessary facts about  Bernoulli  
polynomials and numbers 
can be found, for example, in \cite[Chapter 9]{9}.

For $\bfw= (w_1, \dots, w_d)\in\Rr_{>0}^d$ and $\bfu=(u_1,\dots,u_d)\in\Rr^d$, 
the \textit{multiple Bernoulli polynomials} 
$\BBB_k(t; \bfw,\bfu)$ are defined by the generating function
\begin{equation}
	\left(\prod\nolimits_{j=1}^d\frac{w_j\, s\,\,e^{w_ju_j s}}{e^{w_j s}-1} 
	\right)e^{st}	 = 
	\sum\nolimits_{k\ge 0}\, 
	\BBB_k(t;\bfw,\bfu)\, \frac{s^k}{k!}\, .
	\label{eq2.6}
\end{equation}
We have
\begin{equation}
	\BBB_k(t;\bfw,\bfu)=\sum\nolimits_{n=0}^k\,\frac{k!}{n!(k-n)!}\, 
	\BBB_n(\bfw,\bfu)\, t^{k-n},
	\label{eq2.7}
\end{equation}
where $\BBB_k(\bfw,\bfu)=\BBB_k(0;\bfw,\bfu)$ are the \textit{multiple Bernoulli 
	numbers} defined by the expansion
\begin{equation}
	\prod\nolimits_{j=1}^d\frac{w_j \, s\,\,e^{w_ju_j s}}{e^{w_j s}-1} = 
	\sum\nolimits_{n\ge 
		0}\, 
	\BBB_n(\bfw,\bfu)\, \frac{s^n}{n!},
	\label{eq2.8}
\end{equation}
Replacing $s$ by $w_j\, s$ for $j\in[d]$ in \eqref{eq2.3} and 
substituting these expressions into \eqref{eq2.8}, we find
\begin{equation*}
	\BBB_k(\bfw,\bfu)=\sum_{n_1+\dots +n_d=k}\,\frac{k!}{n_1!\dots n_d!}\, 
	B_{n_1}(u_1)\dots B_{n_d}(u_d)\, w_1^{n_1}\dots w_d^{n_d}.
\end{equation*}

Specializing to $\bfu=\bf{0}$ and $\bfu=\bfe_{1/2}$, we obtain 
two sets of polynomials:
\begin{equation}
	\BBB_k(t; \bfw)=\BBB_k(t; 
	\bfw,{\bf{0}})=\sum\nolimits_{n=0}^k\,\frac{k!}{n!(k-n)!}\, 
	\BBB_n(\bfw)\,\, t^{k-n},
	\label{eq2.10}
\end{equation}
where 
\begin{equation*}
	\BBB_n(\bfw)=\BBB_n(\bfw,{\bf{0}})=\sum_{n_1+\dots 
		+n_d=n}\,\frac{n!}{n_1!\dots n_d!}\, 
	B_{n_1}\dots B_{n_d}\,\, w_1^{n_1}\dots w_d^{n_d}\,,
\end{equation*}
and
\begin{equation}
	\BBB^{\star}_k(t;\bfw)=\BBB_k(t;\bfw, 
	\,\bfe_{1/2})=\sum\nolimits_{n=0}^k\,\frac{k!}{n!(k-n)!}\, 
	\BBB^{\star}_n(\bfw)\, t^{k-n},
	\label{eq2.16}
\end{equation}
where
\begin{align*}
	\BBB^{\star}_n(\bfw)&=\BBB_n(\bfw,\bfe_{1/2})
	\notag
	\\
	&=\sum_{n_1+\dots +n_d=n}\,\frac{n!}{n_1!\dots 
		n_d!}\, 
	B_{n_1}(1/2)\dots B_{n_d}(1/2)\, w_1^{n_1}\dots w_d^{n_d}.
\end{align*}
The values $B_k(1/2)$ of the Bernoulli polynomials $B_k(u)$ at $u=1/2$ are well 
known,
see \cite[Corollary 9.1.5]{9}:
\begin{equation}
	B_k(1/2)=-\left(1-\frac {1}{2^{k-1}}\right)\,B_k,
		\label{eq2.18}
\end{equation}
$\BBB^{\star}_0(1/2)=1,\,\BBB^{\star}_1(1/2)=0,\,\BBB^{\star}_2(1/2)=-\frac{1}{12},
\,\BBB^{\star}_4(1/2)=\frac{7}{240}$, and 
$\BBB^{\star}_k(1/2)=0$ for odd $k\ge 1$, since $B_k =0$ for 
odd $k\ge3$. 

Thus $\BBB^{\star}_n(\bfw)=0$ for odd $n$, and we obtain
\begin{equation}
	\BBB^{\star}_k(t;\bfw)=\sum\nolimits_{0\le 
		n\le\lfloor\frac{k}{2}\rfloor}\,\frac{k!}{(2n)!(k-2n)!}\, 
	\BBB^{\star}_{2n}(\bfw)\,\, t^{k-2n},
	\label{eq2.19}
\end{equation}
where
\begin{align*}
	\BBB^{\star}_{2n}(\bfw)
	=\sum_{n_1+\dots 
		+n_d=n}\,\frac{(2n)!}{(2n_1)!\dots 
		(2n_d)!}\, 
	B_{2n_1}(1/2)\dots B_{2n_d}(1/2)\, w_1^{2n_1}\dots w_d^{2n_d}.
\end{align*}

Using the generating functions for the polynomials 
$\BBB_k(t;\bfw,\bfu),\,\BBB_k(t;\bfw),$ 
and $\BBB^{\star}_k(t;\bfw)$, it is easy to prove the following. 
\begin{lemma}
	The polynomials
	$$\BBB_k(\,t-\bfw\centerdot\bfu;\,\bfw,\bfu\,)$$
	are independent of $\bfu\in\Rr^d$. 
	We also have 
	\begin{equation*}
			\BBB_k(\,t;\bfw,\bfu)=
			\BBB_k(t+\bfw\centerdot\bfu;\,\bfw)=
			\BBB_k^{\star}(\,t+\bfw\centerdot(\bfu-\bfe_{1/2});\,\bfw\,)
	\end{equation*}
		and 
	\begin{align*}
\BBB_k(-t;\bfw)=(-1)^k\,\BBB_k(t+\bfw\centerdot\bfe_1;\bfw)\,,\quad		
	\BBB^{\star}_k(-t;\bfw)=(-1)^k\,\BBB^{\star}_k(t;\bfw)\,.
		\end{align*}
\end{lemma}

It should be noted that the polynomials $\BBB^{\star}_k(t;\bfw)$ are the most 
convenient for writing our 
asymptotic formulas.
In particular,  formula \eqref{eq2.19} for $\BBB^{\star}_k(t;\bfw)$ 
contains $\lfloor\frac{k}{2}\rfloor +1$ terms 
instead of the $k+1$ terms in formulas \eqref{eq2.7} and \eqref{eq2.10} 
for $\BBB_k(t;\bfw,\bfu)$ and $\BBB_k(t;\bfu)$.

We have listed the basic facts about multiple Bernoulli polynomials.
Note that these polynomials, also known in the literature as 
Bernoulli-N\=orlund and Bernoulli-Barnes
polynomials,  arise in various fields of analysis, combinatorics, and number 
theory; see, for example, \cite{1, 3, 11, 11a, 5*, 16**, 17, 11*} 
and the references therein. 
\\
\section{Multiplicative  Diophantine approximations}
\label{sec2}
The Diophantine properties of the problem under consideration are related to 
the arrangement of the hyperplane $\{\bfx :\bfw\centerdot\bfx = 0\}$. 
The hyperplane can be parametrized by the unit normal  vector
$\widehat\bfw=|\bfw|_2^{-1}\bfw\in S^{d-1}_{>0}$,
where $|\cdot|_2$ is the Euclidean norm in $\Rr^d$ and 
$S^{d-1}_{>0}=S^{d-1}\cap\Rr^d_{>0}$ is the ``positive face'' of  the 
standard unit 
sphere  $S^{d-1} \subset \Rr^d$.  
We regard $S^{d-1}_{>0}$ as a probability space with the normalized Lebesgue 
measure.

The hyperplane $\{\bfx :\bfw\centerdot\bfx = 0\}$ can also be  
parametrized  by the vectors 
$\theta_j =w^{-1}_j\bfw
\in\MMM_{j},$ 
where $\MMM_j=\{\bfx\in\Rr^d_{>0}:x_j=1\}\backsimeq\Rr^{d-1}_{>0}$ for
$j\in[d]$. 
These vectors form a $d\times d$ matrix 
$\Theta_{\widehat\bfw} =[\theta_{j,l}]_{j,l=1}^d$ with 1's on the diagonal.
The matrix entries 
$$\theta_{j,l}=\frac{w_l}{w_j}=\frac{\widehat w_l}{\widehat w_j},\quad l\ne j,$$ 
are the \textit{inclines} of the hyperplane 
$\{\bfx :\bfx\centerdot\bfw = 0\}$ to the $j$-th coordinate axis. 
The vectors $\theta_j$ are not independent: 
$\theta_k=\frac{w_j}{w_k}\,\theta_j\, ,$ and any one of them
can be regarded as a parametrization of the 
hyperplane $\{\bfx :\bfw\centerdot\bfx = 0\}$. 
These parametrizations are related by 
\begin{equation*}
	\left.
	\begin{aligned}
			&f_j\colon\widehat\bfw=(\widehat w_1,\dots,\widehat w_d)\in 
			S^{d-1}_{>0}\to		
			\theta_j=\Big(\,\frac{\widehat w_1}{\widehat w_j},\dots,\frac{\widehat 
					w_d}{\widehat w_j}\,\Big)\in \MMM_j\, ,
			\\
			&f_j^{-1}\colon\theta_j\in \MMM_j\to
			\widehat\bfw=|\,\theta_j|_2^{-1}\theta_j\in S^{d-1}_{>0}\, ,		
		\end{aligned}
	\quad\right\}
\end{equation*}
where $f_j$ and $f_j^{-1}$ are $\CCC^{\infty}$ bijections 
(diffeomorphisms between the manifolds $S^{d-1}_{>0}$ and $\MMM_j$).

\textbf{Definition 2.1.}
For $\kappa\ge0$, a set of numbers 
$\theta = (\theta_1,\dots,\theta_{d-1})\in\Rr^{d-1}$ is called
$\kappa$-\textit{multiplicatively approximable} if the inequality
\begin{equation}
	m^{1+\kappa}\,\prod\nolimits_{l\in [d-1]}\,\, \langle 
	\theta_{l}\,\, m\rangle >c(\theta) \,,
	\label{eq1.18*}
\end{equation}
holds for all $m\in \Zz_{>0}$ with some constant $c(\theta)>0$.

If the inequality \eqref{eq1.18*} holds for arbitrarily small $\kappa >0$,
the set  $(\theta_1,\dots,\theta_{d-1})$ is called \textit{badly 
	multiplicatively approximable}.

\textbf{Remarks.}
\textit{(i)} It follows from the metric theory of Diophantine 
approximation that almost all 
$\theta = (\theta_1,\dots,\theta_{d-1})\in\Rr^{d-1}$ 
are badly multiplicatively approximable; see \cite[Chap.1, Sec. 8]{25} 
and the references therein.

\textit{(ii)}
It follows from  Schmidt's theorem on Diophantine approximation that a set
$(\theta_1,\dots,\theta_{d-1})$ is
badly multiplicatively approximable if $1,\theta_1,\dots,\theta_{d-1}$ 
are real algebraic numbers linearly independent over $\Qq$ (see \cite[Chap.VI, 
Thm.1B]{19}).  

\textit{(iii)}
For $d=2$, there are irrational numbers $\theta_1$ (quadratic irrationalities, for 
example) such that \eqref{eq1.18*} 
holds with $\kappa=0$.
At the same time, a famous conjecture of Littlewood states that  
there are no real numbers $\theta_1,\dots,\theta_{d-1}$ for $d\ge3$ satisfying 
\eqref{eq1.18*} with 
$\kappa =0$ (see \cite[Chap.V, Sec. 10.3] {4**} and \cite [Sec. 30.3]{26}).
In the book \cite [Chap.IV]{26}, the reader will  find a discussion of  
multiplicative Diophantine approximation in the context of various problems in 
number theory.
\\

We extend Definition~2.1 to the set of inclines $\Theta_{\widehat\bfw}$.

\textbf{Definition 2.2.}
For $\kappa > 0$, the inclines $\Theta_{\widehat\bfw}$ are called 
$\kappa$-\textit{multiplicatively approximable} if the inequalities
\begin{equation}
	m^{1+\kappa}\,\prod\nolimits_{l\in [d],\,l\ne j}\,\, \langle 
	\theta_{j,l}\,\, m\rangle >c(\Theta_{\widehat\bfw}), 
	\quad j\in [d]\,,
	\label{eq1.18}
\end{equation}
hold  for all $m\in \Zz_{>0}$ with some 
$c(\Theta_{\widehat\bfw})>0$ simultaneously for all $j\in [d]$.

If the inequalities \eqref{eq1.18} hold with arbitrarily small $\kappa >0$,
the inclines $\Theta_{\widehat\bfw}$ are called \textit{badly multiplicatively 
	approximable}.

\textbf{Remarks.}
\textit{(i)}
Based on the results of the metric theory of Diophantine approximation,
we will see that for almost all $\widehat\bfw\in S^{d-1}_{>0}$, 
the inclines $\Theta_{\widehat\bfw}$ are badly multiplicatively 
approximable (see Proposition~2.1 below).

\textit{(ii)}
It  follows from
Schmidt's theorem cited above  
that the inclines $\Theta_{\widehat\bfw}$ are
badly multiplicatively  approximable if $w_1,\dots,w_d$ are
real algebraic numbers linearly independent over $\Qq$. 
\\

For inclines $\Theta_{\widehat\bfw} =[\theta_{j,l}]_{j,l=1}^d$ with irrational 
$\theta_{j,l},\,l\ne j$, we consider the following sums:
\begin{equation}
S_0(\Theta_{\widehat\bfw}, T)=\sum_{j=1}^d\, S_0(\theta_j, T),\quad
S(\Theta_{\widehat\bfw}, T)=\sum_{j=1}^d\, S(\theta_j, T)\,,	
	\label{eq1D}
	\end{equation}
where	
\begin{equation*}
	\left.
	\begin{aligned}
		S_0(\theta_j, T)&=
		\sum_{m=1}^T
		\,\,
		\frac{1}{m\,\,\prod_{l\in[d],l\ne j}|\sin(\pi\,\theta_{j,l}\, m)|}\,,
		\\
		S(\theta_j, T)&=
		\sum_{m=1}^T\,\,
		\frac{1}{m\,\prod_{l\in[d],l\ne j}\langle\,\theta_{j,l}\,m\, \rangle}\,,
		\end{aligned}
	\quad\right\}
\end{equation*}
These sums are equivalent: $S(\Theta_{\widehat\bfw}, T)\approx  
S_0(\Theta_{\widehat\bfw}, T) $, or, more precisely,
\begin{equation*}
\pi^{1-d}\,	S(\Theta_{\widehat\bfw}, T)\le
S_0(\Theta_{\widehat\bfw}, T)\le 
2^{1-d}\,S(\Theta_{\widehat\bfw}, T)\,,
	\end{equation*}
since $2\langle x\rangle \le |\sin\pi x|\le \pi \langle x\rangle$.
Both  of these sums are used in the paper.
\begin{proposition}
For almost all $\widehat\bfw\in S^{d-1}_{>0}$, 
the inclines $\Theta_{\widehat\bfw}$ are badly multiplicatively approximable and
\begin{equation*}
S(\Theta_{\widehat\bfw}, T)=
O_{\widehat\bfw,\epsilon}\,\big(\,(\log\,T)^{d+\epsilon}\,
\big)
\end{equation*}
for any $\epsilon >0$.
%
\end{proposition}
\begin{proof}
It was proved in  \cite[Lemma I]{24} that
for a given $j\in [d]$ and for almost all 
$(\theta_{j,1},\dots,\theta_{j,d})\in\MMM_j$, 
\begin{equation}
\sum_{m\ge 1}\,\,
\frac{1}{m\,\big(\log (m+1)\big)^{d+\epsilon}\,\,\prod_{l\in[d],l\ne 
j}|\sin(\pi\,\theta_{j,l}\, m)|} \le 
C_{\widehat\bfw,\epsilon} <\infty
\label{eq2D}	
\end{equation}
for any $\epsilon>0$.

Let $\EEE_j\subset\MMM_j$ denote the null set where the series \eqref{eq2D} 
diverges. Then $f_j^{-1}(\EEE_j)\subset S^{d-1}_{>0}$ is a null set, since
$f_j^{-1}$ is a $\CCC^{\infty}$ mapping. Let 
$\EEE=\bigcup_{j\in [d]}\,f_j^{-1}(\EEE_j)$. Then 
$f_j(\EEE)\subset\MMM_j$ for $j\in[d]$ are  null 
subsets, since
$f_j$ are $\CCC^{\infty}$ mappings. 
Therefore, for  $\widehat\bfw\in S^{d-1}_{>0}\setminus\EEE$	and 
$(\theta_{j,1},\dots,
\theta_{j,d})=f_j(\widehat\bfw)\in\MMM_j$, the series \eqref{eq2D} converge
simultaneously   for all $j\in[d]$. This implies that 
for $\widehat\bfw\in S^{d-1}_{>0}\setminus\EEE$,
the inequalities
\begin{equation*}
	m\,\big(\log (m+1)\big)^{d+\epsilon}\,\prod\nolimits_{l\in [d],\,l\ne j}\,\, 
	\langle 
	\theta_{j,l}\,\, m\rangle >c_{\widehat\bfw,\epsilon}(\Theta), 
	\quad j\in [d]\,,
\end{equation*}
hold  for all $m\in \Zz_{>0}$ with some 
$c_{\widehat\bfw,\epsilon}>0$ simultaneously for all $j\in [d]$.
Therefore, 
for almost all $\widehat\bfw\in S^{d-1}_{>0}$, 
the inclines $\Theta_{\widehat\bfw}$ are badly multiplicatively approximable. 
Furthermore, for  $\widehat\bfw\in S^{d-1}_{>0}\setminus\EEE$,
we also have
 \begin{align*}
	S(\Theta_{\widehat\bfw},T)&\ll S_0(\Theta_{\widehat\bfw},T)
	\\
	&\ll\sum_{j=1}^d\,\, (\log\,T)^{d+\epsilon} \sum_{m=1}^T\,\,
\frac{1}{m\,|\log (m+1)|^{d+\epsilon}\,\prod_{l\in[d],l\ne 
j}|\sin(\pi\,\theta_{j,l}\, m)|}
\\
&\ll 
\, (\log\,T)^{d+\epsilon}\,. 	
 	\end{align*}
 This completes the proof.	
 	 	\end{proof}


The following Proposition~3.2 contains a multidimensional generalization, 
for $d>2$, 
of the well-known bounds for $d=2$; see, for example, \cite[Chap.2, Lemma 
3.3]{2A}, 
\cite[Chap.III]{3A}.
\begin{proposition}
For $\kappa$-multiplicatively approximable inclines 
$\Theta_{\widehat\bfw}$ with $\kappa>0$, 
\begin{equation*}
	S(\Theta_{\widehat\bfw}, T)=O_{\bfw}\,\big(\, T^{\kappa} 
		\,(\log T)^{d-2}\,\big)\, , 
\end{equation*}
and for  badly multiplicatively approximable inclines $\Theta_{\widehat\bfw}$,
\begin{equation}
	S(\Theta_{\widehat\bfw}, T)=O_{\bfw,\epsilon}\big(\, T^{\epsilon} \,\big)
	\label{eq4D}
\end{equation}
for any $\epsilon >0$.
In particular,  \eqref{eq4D} holds if $w_1, 
\dots, w_d$ are positive algebraic numbers linearly independent over $\Qq$.
\end{proposition}
\begin{proof}
Let
$\theta = (\theta_1,\dots,\theta_{d-1})\in\Rr^{d-1}$ be irrational numbers.
Consider the sum
\begin{equation}
s(\theta, m)=
\sum_{k=1}^m\,\,
\frac{1}{\langle\,\theta_1\,k\, \rangle\dots\langle\,\theta_{d-1}\,k\, \rangle}\,.
\label{eq5D}
\end{equation}
\begin{lemma}
For a set  
$\theta = (\theta_1,\dots,\theta_{d-1})\in\Rr^{d-1}$ of
$\kappa$-multiplicatively approximable numbers with $\kappa>0$,
\begin{equation*}
s(\theta, m)=O_{\theta}\,\big(\, m^{\kappa + 1} 
\,(\log m)^{d-2}\,\big)\,.	
	\end{equation*}	
	\end{lemma}
	\begin{proof}
We will need the following notation. For $L=(l_1,\dots,l_{d-1})
\in\Zz_{\ge 0}^{d-1}$, we set
\begin{align*}
	\frak{M}(L)=\{\,k\in\Zz_{> 0}: \langle\,\theta_j\,k\, \rangle\le 
	2^{-l_j-1},\,j\in[d-1]\,\}\,
\end{align*}
and
\begin{equation*}
\mu(L)=\min\{\,k: k\in\frak{M}(L)\,\}\,.	
	\end{equation*}				
We have				
\begin{align*}
	\mu(L)^{1+\kappa}\,\prod\nolimits_{j\in [d-1]}2^{-l_j-1}\ge
	\mu(L)^{1+\kappa}\,
	\prod\nolimits_{j\in [d-1]}\,\, \langle \, 
	\theta_{j}\,\, \mu(L)\,\rangle \ge c\,(\theta)\, . 
\end{align*}
Hence, 
\begin{equation*}
\mu(L)\ge c_1\,2^{\frac{1}{\kappa 
		+1}\, |L|_1}
	\end{equation*}
with $c_1=\big(c\,(\theta)\,2^{d-1}\big)^{\frac{1}{1+\kappa}}$ and  
$|L|_1=l_1+\cdots +l_{d-1}$.

We set
\begin{align*}
	\frak{N}(L)=\{\,k\in\Zz_{\ge 0}: 2^{-l_j-2}<
	\langle\,\theta_j\,k\, \rangle\le 
	2^{-l_j-1},\,j\in[d-1]\,\}\,.
\end{align*}		
Write 
$\langle\,\theta_j\,k\, \rangle=(-1)^{\sigma_j}\,(\,\theta_j\,k +n_j(k)\,)$,
where $\sigma_j\in\{0,1\}$ and the integers $n_j(k)$ are determined uniquely, since
$\theta_j$ for $j\in[d-1]$	are irrational. 
With this notation,
for a given set $\sigma=(\sigma_1,\dots,\sigma_{d-1})\in\{0,1\}^{d-1}$, we set
\begin{align*}
	\frak{N}^{\,\sigma}(L)=\{\,k\in\frak{N}(L) :
\langle\,\theta_j\,k\, \rangle
=(-1)^{\sigma_j}\,(\,\theta_j\,k +n_j(k)\,),\,j\in[d-1]\,\}.	
\end{align*}
These sets form the  partitions
\begin{equation}
\Zz_{>0}=\bigsqcup\nolimits_{\,L\in\Zz^{d-1}_{\ge0}}\,
\bigsqcup\nolimits_{\,\sigma\in\{0,1\}^{d-1}} \,\frak{N}^{\,\sigma}(L)\,.
\label{eq7D}
\end{equation}
Clearly, 
\begin{equation*}
\mu^{\,\sigma}(L)=\min\{\,k: k\in\frak{N}^{\,\sigma}(L)\,\}\ge\mu(L)\,,
\end{equation*}
since $\frak{N}^{\,\sigma}(L)\subset\frak{M}(L)$.

Let integers $0<k_1<k_2$ belong to $\frak{N}^{\,\sigma}(L)$. Then
$|\langle\,\theta_j\,k_2\, \rangle - \langle\,\theta_j\,k_1\, \rangle | 
<2^{-l_j-1}$.
On the other hand, 
$\langle\,\theta_j\,k_2\, \rangle - \langle\,\theta_j\,k_1\, \rangle =
(-1)^{\sigma_j}\,\big(\,\theta_j\,(k_2-k_1) +n_j(k_2)-n_j(k_1)\,\big)$,
and we can write
$$2^{-l_j-1}>|\langle\,\theta_j\,k_2\, \rangle - \langle\,\theta_j\,k_1\, \rangle | 
=|\theta_j\,(k_2-k_1) +n_j(k_2)-n_j(k_1)|
\ge \langle\,\theta_j\,(k_2-k_1)\, \rangle .$$
Therefore, $k_2-k_1 \in\frak{M}(L)$ and 
$$k_2-k_1 \ge \mu(L)\ge c_1\,2^{\frac{1}{\kappa +1}\, |L|_1}.$$	
Thus, the integers $k\in\frak{N}^{\,\sigma}(L)$	are well separated, and we have
\begin{equation*}
\#\{k\in\frak{N}^{\,\sigma}(L): k\le m\}
\le\big\lfloor \mu(L)^{-1} m\rfloor
\,\le\, c_1^{-1}\,\,2^{-\frac{1}{\kappa +1}\, |L|_1} \,m\,.	
	\end{equation*}
In addition, $\#\{k\in\frak{N}^{\,\sigma}(L): k\le m\}=0$ if $m<\mu(L)$.
Hence, $\#\{k\in\frak{N}^{\,\sigma}(L): k\le m\}=0$ for $|L|_1>\lambda(m)$, 
where $\lambda(m) =(\kappa +1)\log m +O(1)$ for large $m$.

Using the partition \eqref{eq7D}, we  write the sum \eqref{eq5D} as follows:
\begin{equation}
	s(\theta, m)=\sum\nolimits_{\,L\in\Zz^{d-1}_{\ge0}}\,
	\sum\nolimits_{\,\sigma\in[0,1]^{d-1}} \,s^{\sigma}_L (m)\,,
	\label{eq8D}
\end{equation}
where
\begin{equation*}
	s^{\sigma}_L ( m)=
	\sum\nolimits_{k\in\frak{N}^{\,\sigma}(L):\, k\le m}\,\,
	\frac{1}{\langle\,\theta_1\,k\, \rangle\dots\langle\,\theta_{d-1}\,k\, 
	\rangle}\,.
\end{equation*}
For this sum, we have
\begin{equation*}
	s^{\sigma}_L ( m)< 2^{2(d-1)+|L|_1}\,\#\{k\in\frak{N}^{\,\sigma}(L): k\le m\}
	\le c_2\,\,2^{\frac{\kappa}{\kappa +1}\, |L|_1} \,m\,
\end{equation*}
and $s^{\sigma}_L ( m)=0$ for $|L|_1>\lambda(m)$.

Substituting this bound into \eqref{eq8D}, we obtain
\begin{equation*}
	s(\theta, m)\le c_2\, m\,\sum\nolimits_{\,\ell\le\lambda(m)}\, 
	\,\,2^{\frac{\kappa}{\kappa +1}\, \ell}\,\,
	\#\{L\in\Zz^{d-1}_{\ge 0}: |L|_1=\ell\}.
\end{equation*}
Clearly, $\#\{L\in\Zz^{d-1}_{\ge 0}: |L|_1=\ell\}\ll \ell^{\,d-2}$. 
Therefore,
\begin{align*}
	s(\theta, m)&\ll\, m\,\sum\nolimits_{\,\ell\le\lambda(m)}\, 
	\,\,2^{\frac{\kappa}{\kappa +1}\, \ell}\,\,\ell^{\,d-2}
	\\
	&\ll m\, 2^{\frac{\kappa}{\kappa +1}\, \lambda(m)}\,\,\lambda(m)^{\,d-2}\ll
	 m^{\kappa + 1} \,(\log m)^{d-2}.
\end{align*}
This proves Lemma~3.1.
\end{proof}

By partial summation, we have
\begin{align*}
S(\theta, T)&=
\sum_{m=1}^T\,\,
\frac{1}{m \,\langle\,\theta_1\,k\, \rangle\dots\langle\,\theta_{d-1}\,k\, \rangle}
=\sum_{m=1}^T\,\frac{s(\theta,m)}{m(m+1)} +\frac{s(\theta,m)}{m+1}\,.
\end{align*}
Using Lemma~3.1, we obtain
\begin{equation*}
	S(\theta, T)\ll\sum_{m=1}^T\, m^{\kappa -1}\,\big(\log (m+1) \big)^{d-2}
	=O_{\theta}\,\big(\, T^{\kappa} 
	\,(\log T)^{d-2}\,\big)\, , 
\end{equation*}
since $\kappa>0$.
Substituting this bound with $\theta=\theta_j$ for $j\in[d-1]$ into \eqref{eq1D},
we complete the proof of Proposition~3.2. 
\end{proof}


\section{Main results}
\label{sec4}


We consider the number of 
integer points in shifted simplices:
\begin{equation}
	N(t;\bfw,\bfu)=\#\,\{\,\Delta^{-} (t;\bfw,  
	\bfu)\,\cap\,\Zz^d\,\}, 
	\label{eq1.22}
\end{equation}
where $\Delta^{-} (t;\bfw,\bfu)=\Delta^{-} (t;\bfw)+\bfu$ is the simplex 
$\Delta^{-} (t;\bfw)$ shifted by $\bfu\in\Rr^d$: 
\begin{align*}
	\Delta^{-} (t;\bfw,\bfu)
	=\{\bfx:x_j>u_j,\, j\in [d], 
	\,\,\bfw\centerdot\bfx<t+\bfw\centerdot\bfu\}.
\end{align*}
Clearly, $N(t;\bfw,\bfu)$ is a periodic piecewise constant function: 
$N(t;\bfw,\bfu 
+\bfz)=N(t;\bfw,\bfu)$ for $\bfz\in\Zz^d$. 
As a function of $t\in\Rr_{>0}$, $N(t;\bfw,\bfu)$ is  piecewise constant,
monotone non-decreasing, and left-continuous.

For $1/2>\delta >0$,  define a subset $\UUU(\delta)\subset\Rr^d$ by 
\begin{align*}
	\UUU(\delta)&=\{\,\bfu=(u_1, \dots, u_d)\in\Rr^d: \langle 
	u_j\rangle\ge\delta, 
	\,j\in [d]\,\}\notag
	\\
	&=\bigcup\nolimits_{\bfm\in\Zz^d}\, \left\{ K(\delta)
	+\bfm \,\right\},
\end{align*}
where $K(\delta)=\left[\,\delta, 1-\delta\,\right]^d\subset\Rr^d $ is a small cube 
centered at $\bfe_{1/2}$.

We write $((u))=\{u\}-1/2$,  where $\{u\}$ is the fractional part of 
$u\in\Rr$, and for $\bfu =(u_1, \dots, u_d)\in\Rr^d$, we write
\begin{equation*}
	((\bfu))=\big(((u_1)), \dots, ((u_d))\big),\quad
	\{\bfu\}=\big(\{u_1\}, \dots, \{u_d\}\big)\, \in\,\Rr^d\,.	
\end{equation*} 

We define a periodization of the multiple Bernoulli polynomials	by 
\begin{equation}
	\BBB_k(\,t;\bfw,\{\bfu\})=
	\BBB_k(t+\bfw\centerdot\{\bfu\};\,\bfw)=
	\BBB_k^{\star}(\,t+\bfw\centerdot((\bfu));\,\bfw\,)\,,
	\label{eq2.22a}
\end{equation}
(see Lemma 2.1).
We write  
$f(t\mp 0)=\lim\nolimits\,_{\alpha\downarrow 0}\, f(t\mp\alpha)$
for a piecewise continuous function $f$. 

\begin{theorem}
	[ \bf{The Main Theorem} ]
Suppose that the inclines  $\Theta_{\widehat\bfw}$ are $\kappa$-multiplicatively 
approximable. Then 
	\begin{equation}	
		N(t;\bfw,\bfu)=
		\LLLL_d(t;\bfw,\{\bfu\})+R(t;\bfw,\{\bfu\})\,,
		\label{eq3.2*c}	
	\end{equation}
	where  the \textbf{leading term} is defined by
	\begin{align*}
		\LLLL_d(t;\bfw,\{\bfu\})=\frac{1}{d!\,w_1\dots w_d}\,\,
		\BBB^{\star}_d\,(t+\bfw\centerdot((\bfu));\,\bfw)\,,	
	\end{align*}
	and the \textbf{ error}
	\begin{equation}
		\RRR_{\delta}(t;\bfw)=\sup\nolimits_{\,\bfu\,\in\,\UUU(\delta)}\, 
		\RRR(t;\bfw,\{\bfu\})\, ,
		\label{eq3.28**cc}
	\end{equation}
	where
	\begin{equation*}
		\RRR(t;\bfw,\{\bfu\})=\max\{\,|R(t-0;\bfw,\{\bfu\})|,\,
		|R(t+0;\bfw,\{\bfu\})|\,\}\,, 
	\end{equation*}
	satisfies the bound 
\begin{align}
	\RRR_{\delta}(t;\bfw)=
O_{\widehat\bfw,\delta,\epsilon}\Big(	\,S\,(\Theta_{\widehat\bfw}, 
T^{1+\epsilon})+\, \,T^{-1}\,t^{d-1} \,\Big)\, , 
	\label{eq1.31aa}
\end{align}		
for any $\epsilon>0$. Here 
$S(\Theta_{\widehat\bfw}, T)$
is defined in  \eqref{eq1D}, and $T$ is an arbitrarily large parameter.
\end{theorem}

The proof of this theorem is the subject of the next sections; for now, we wish to 
derive its corollaries. 

\begin{theorem}[ \bf{The Main Asymptotic Bounds} ]
	The error term \eqref{eq3.28**cc}  satisfies the following bounds:

\textit{(i)} For almost all $\widehat\bfw\in S^{d-1}_{>0}$,  
\begin{equation}
	\RRR_{\delta}(t;\bfw)= 
	O_{\widehat\bfw,\delta,\epsilon}\big(\,(\log\,t)^{d+\epsilon}\big). 
	\label{eq3.30q}
\end{equation}	
for any $\epsilon >0$.	

\textit{(ii)} For $\kappa$-multiplicatively approximable inclines 
$\Theta_{\widehat\bfw}$ with $\kappa>0$, 
\begin{equation}
	\RRR_{\delta}(t;\bfw)= O_{\delta,\epsilon}\big(\, 
	t^{\frac{\kappa}{1+\kappa} (d-1)+\epsilon}\,\big)\, , 
	\label{eq1.31*}
\end{equation}
for any $\epsilon>0$,
and for  badly multiplicatively approximable inclines $\Theta_{\widehat\bfw}$,
\begin{equation}
	\RRR_{\delta}(t;\bfw)= O_{\delta,\epsilon}(\, t^{\epsilon}) 
	\label{eq1.31}
\end{equation}
for any $\epsilon >0$.
In particular, \eqref{eq1.31} holds 
if $w_1, \dots, w_d$ are positive algebraic numbers 
linearly independent over $\Qq$.	
	\end{theorem}
\begin{proof}
\textit{(i)}
Using Theorem~4.1 and Proposition~3.1, we find 
\begin{equation*}
	\RRR_{\delta}(t;\bfw)\,
	\,\ll\,(\,\log T\,)^{d+\epsilon}
	+T^{-1}t^{d-1}\, ,
\end{equation*}
Taking  $T=t^{d-1}$ here, we obtain \eqref{eq3.30q}.

\textit{(ii)}
Using Theorem~4.1 and Proposition~3.2, we find
\begin{equation*}
	\RRR_{\delta}(t;\bfw)\,\ll\,
	T^{\kappa+\epsilon}\,(\log T)^{d-1}+T^{-1}t^{d-1}
	\ll\,(T^{\kappa}\,+T^{-1}t^{d-1})\, T^{\epsilon} \, .
\end{equation*}
To minimize the sum of the two terms in this estimate, we set 
$T^{\kappa}\,=\,T^{-1}t^{d-1}$. This yields
$T=t^{\frac{d-1}{1+\kappa}}$. Therefore,
$	\RRR_{\delta}(t;\bfw)\,\ll\,
	t^{\frac{\kappa}{1+\kappa}(d-1)+\epsilon} $
and \eqref{eq1.31*} follows.
\end{proof}

We now specialize the shifts $\bfu\in\Rr^d$ in Theorem~4.2 to estimate 
$N^{\mp}(t;\bfw)$.
\begin{theorem}[ \bf{The Hardy--Littlewood Asymptotics in dimension $d$} ]
	We have
	\begin{equation}	
		N^{\mp}(t;\bfw)=
		\LLLL_d^{\,\mp}(t;\bfw)+R^{\mp}(t;\bfw)\,,
		\label{eq1.2*c}	
	\end{equation}
	where  the \textbf{leading terms} are defined by
	\begin{align*}
		\LLLL_d^{\mp}(t;\bfw)=\frac{1}{d!\,w_1\dots w_d}\,\,
		\BBB_d^{\star}\,(t\mp \bfw\centerdot\bfe_{1/2};\,\bfw)\,,	
	\end{align*} 
	and the \textbf{errors}
	\begin{align*}
		\RRR^{\mp}(t;\bfw)=\max\{\,|R^{\mp}(t-0;\bfw)|,\,|R^{\mp}(t+0;\bfw)|\,\},
	\end{align*}
	satisfy the bounds:
	
	\textit{(i)} 
	For almost all $\widehat\bfw\in S^{d-1}_{>0}$, 
	\begin{equation*}
		\RRR^{\mp}(t;\bfw)=O_{\widehat\bfw,\epsilon}\,\big(\,(\log\,t)^{d+\epsilon}\,
		\big)
	\end{equation*}
	for any $\epsilon >0$.
	
	\textit{(ii)} For $\kappa$-multiplicatively approximable inclines 
	$\Theta_{\widehat\bfw}$ with $\kappa>0$, 
	\begin{equation}
		\RRR^{\mp}(t;\bfw)=O_{\bfw}\,\big(\, t^{\frac{\kappa}{1+\kappa} 
			(d-1)+\epsilon}\,\big)\, , 
		\label{eq1.3b}
	\end{equation}
	for any $\epsilon>0$, 
	and for  badly multiplicatively approximable inclines $\Theta_{\widehat\bfw}$,
	\begin{equation}
		\RRR^{\mp}(t;\bfw)=O_{\bfw,\epsilon}\big(\, t^{\epsilon} \,\big)
		\label{eq1.3c}
	\end{equation}
	for any $\epsilon >0$.
	In particular,  \eqref{eq1.3c} holds if $w_1, 
	\dots, w_d$ are positive algebraic numbers linearly independent over $\Qq$.
\end{theorem}

Before proving this result, we need to carefully consider the relationship 
between  $N(t;\bfw,\bfu)$ and $N^{\mp}(t;\bfw)$.
Let us put
\begin{equation*}
	\tau (t)=\#\,(\bfx\in\Zz^d_{>0}:\,\bfw\centerdot\bfx=t),\quad t>0.
	\end{equation*}
Clearly, $\tau (t)=0$  except when
$t=t_{\bfm}=\bfw\centerdot\bfm$ for $\bfm\in\Zz^d_{>0}$.
\begin{lemma}
	For each $t\in\Rr_{>0}$, we have  $\tau(t\mp 0)=0$
	and $\tau(t)\le 2\RRR^{-}(t;\bfw)$.
\end{lemma}
\begin{proof}
	For each 
	$t\in\Rr_{>0}$ and all sufficiently small $\alpha>0,\, \tau(t\mp\alpha)=0$.
	This proves the first equality. We have
	$$\tau(t)=N^{-}(t + 0)-N^{-}(t - 0)
	=R^{-}(t + 0) - R^{-}(t - 0).$$
	and the second inequality follows since $\tau(t)\ge 0$.
\end{proof}

\begin{proposition}[\bf{The Invariance Principle}]
	For $t>\bfw\centerdot\bfe_1$, the quantity
	\begin{align}
		N(t-\bfw\centerdot\{\bfu\};\bfw,\{\bfu\})
		\label{eq4.23a}
	\end{align}
	is independent of $\bfu\in\UUU(\delta)$, and we have 
\begin{equation}
	\left.
	\begin{aligned}	
		&N^{-}(t;\bfw)=N(t-\bfw\centerdot\{\bfu\};\bfw,\{\bfu\})=
		N^{+}(t-\bfw\centerdot\bfe_1;\bfw)-\tau (t)\,.	
\\
&R^{-}(t;\bfw)\,
=\,R(t-\bfw\centerdot\{\bfu\};\bfw,\{\bfu\})
=\,R^{+}(t-\bfw\centerdot\bfe_1;\bfw)-\tau (t)\,.	
\\
	&\RRR^{-}(t;\bfw)\,
	=\RRR(t-\bfw\centerdot\{\bfu\};\bfw,\{\bfu\})
	=
	\RRR^{+}(t-\bfw\centerdot\bfe_1;\bfw)\,.	
\end{aligned}
\quad\right\}
\label{eq4.23abbb}
\end{equation}	
	\end{proposition}
	\begin{proof}
		We have
\begin{align*}
	N(t-\bfw\centerdot\{\bfu\};\bfw,\{\bfu\})
=\#\,\{\bfx\in\Zz^d: \{u_j\}<x_j,\,j\in[d],\,\,\bfw\centerdot\bfx<t \}.	
	\label{eq4.23c}
\end{align*}
This shows that the quantity \eqref{eq4.23a} is independent of 
$\bfu\in\UUU(\delta)$, since the hyperplanes 
$\{\bfx: x_j=\{u_j\}\}$ for $j\in[d]$ and $\bfu\in\UUU(\delta)$  contain no 
integer points. 
This also proves the first line of equalities in \eqref{eq4.23abbb}.
	
	Substituting the asymptotics \eqref{eq1.2*c} and \eqref{eq3.2*c} into 
	the first line of equalities in \eqref{eq4.23abbb}, we see
	 that the leading terms cancel by \eqref{eq2.22a}, and we obtain the 
second line in \eqref{eq4.23abbb}.
	Using Lemma~4.1, we find
	\begin{align*}
		R^{-}(t\mp 0;\bfw)
		=
		R(t-\bfw\centerdot\{\bfu\}\mp 0;\bfw,\{\bfu\})
		=
		R^{+}(t-\bfw\centerdot\bfe_1\mp 0;\bfw)	
	\end{align*}
	and the third line in \eqref{eq4.23abbb} follows.
\end{proof}

	\begin{proof}[\bf{Proof of Theorem~4.3.}]
		We prove that Theorems 4.2 and 4.3 are equivalent.
		Setting $\bfu=\bfe_{1/2}\in\UUU(\delta)$ in \eqref{eq4.23abbb}, we conclude 
		that Theorem~4.2 implies Theorem~4.3. On the other hand,
	\begin{align*}
		\RRR(t;\bfw,\{\bfu\})
		=\RRR^{-}(t+\bfw\centerdot\{\bfu\};\bfw)
		=\RRR^{+}(t-\bfw\centerdot(\bfe_1-\{\bfu\});\bfw)	
	\end{align*}
	by the third line in \eqref{eq4.23abbb}.
	This proves that Theorem~4.2 follows from Theorem~4.3.
	This also shows that the bounds for $\RRR^{-}(t;\bfw)$ and $\RRR^{+}(t;\bfw)$ 
	in Theorem~4.3 are equivalent to each other.
	\end{proof}

\enlargethispage{4\baselineskip}

\textbf{Examples.}
We conclude our discussion of the results with examples  showing that the exponent
$\frac{\kappa}{1+\kappa} (d-1)$ in the bounds \eqref{eq1.3b} and \eqref{eq1.31*}  
in Theorems 4.2 and 4.3 is  best possible.
The examples are based on a well-known construction from coding theory  
(BCH codes); see, for 
example, \cite[Sec. 11]{14ab}. The only difference is that, instead of finite 
fields as in coding theory, we consider algebraic number fields.

Let $\Ff$ be a real algebraic number field of degree $q=[\Ff:\Qq]$.
Choose positive numbers $\alpha_1,\dots,\alpha_q$ in $\Ff$ linearly independent
over $\Qq$ and define the sequence
\begin{equation}
	w_j=\alpha_1 + j\,\alpha_2 +\cdots + j^{q-1}\alpha_q, \quad j\in\Zz_{>0}.
	\label{eq4a}
\end{equation}
Now define the weights $\bfw =(w_1,\dots,w_d)$ in dimension 
$d=p(q-1)+1$ for $p\in\Zz_{>0}$.
\begin{lemma}
	With the above notation, we have the following: 

(i) Any $q$ terms in the sequence \eqref{eq4a} are linearly independent over $\Qq$.
	
(ii) $\Gamma=\{\bfx\in\Zz^d: \bfw\centerdot\bfx=0\}$
is a sublattice (subgroup of $\Zz^d$) of rank $d-q=(p-1)(q-1)$.
\end{lemma}
\begin{proof}\textit{(i)}
Write the equation $\sum\nolimits_{i=1}^q\,c_i\,w_{j_i}=0$
with $j_1<\cdots<j_q$ and rational $c_1,\dots,c_q$.
Since $\alpha_1,\dots,\alpha_q$ are linearly independent over $\Qq$,
the equation can be separated into $q$ equations in $q$ unknowns:
$\sum\nolimits_{i=1}^q\,j_i^k\,c_i=0$ for $k=0,\dots,q-1$. 
The matrix of these equations is a nonzero Vandermonde matrix. 
This proves the first part of the lemma.

\textit{(i)}
Write the equation $\sum\nolimits_{j=1}^d\,c_j\,w_j=0$ with rational
$c_1,\dots,c_d$.
Taking into account the first part of the lemma, we conclude that
the equation can be separated into $q$ equations in $d$ unknowns:
$\sum\nolimits_{j=1}^d\,j^k\,c_j=0$ for $k=0,\dots,q-1$. Any  $d-q$ 
coefficients $c_j$ for $j\in[d]$ can be chosen arbitrarily, 
and the remaining $q$ coefficients are then uniquely determined by the equations.
This proves that $G=\{\bfc\in\Qq^d: \bfw\centerdot\bfc=0\}$ is a rational 
subspace of dimension $d-q$. Clearly, for any rational subspace 
$V\subset\Rr^d$ of dimension $s$, the intersection $V\cap\Zz^d$ is a subgroup of 
$\Zz^d$ of rank $s$. This proves the second part of the lemma.
\end{proof}

We now consider the simplices $\Delta^{-}(t;\bfw)\subset\Rr^d$ with  weights
$\bfw =(w_1,\dots,w_d)$ defined in \eqref{eq4a}. 
The corresponding inclines $\theta_{j,l}=\frac{w_l}{w_j}$ for $l,j\in[d],$ are 
algebraic numbers in the field $\Ff$. For each $j\in[d]$, we introduce a partition 
$$[d]\setminus\{j\}=\bigsqcup\nolimits_{\,i=1}^{\, 
p}\,E_i(j),\quad\#\{E_i(j)\}=q-1\,. $$
By Lemma~4.2\textit{(i)}, for $j\in[d]$ and $i\in[p]$, the numbers 
$1,\,\theta_{j,l}$ for $l\in E_i(j)$ 
are real algebraic numbers linearly independent over $\Qq$.
Therefore, by Schmidt's theorem, we have
\begin{equation*}
	m^{1+\epsilon_0}\,\prod\nolimits_{l\in E_i(j)}\,\, \langle 
	\theta_{j,l}\,\, m\rangle >c, 
	\quad j\in [d],\quad i\in[p],
\end{equation*}
for any $\epsilon_0>0$ and some $c>0$ (if $q=2$, one may take $\epsilon_0=0$).
Multiplying these inequalities over all $i\in[p]$, we obtain
\begin{equation*}
	m^{p+p\epsilon_0}\,\prod\nolimits_{l\in [d],\,l\ne j}\,\, \langle 
	\theta_{j,l}\,\, m\rangle >c^p, 
	\quad j\in [d]\,,
\end{equation*}
By Definition~2.2, this means that the inclines $\Theta_{\widehat\bfw}$ are
$\kappa$-multiplicatively approximable  with $\kappa=p-1+p\epsilon_0$.

Theorem~4.3 now implies the upper bound 
\begin{equation*}
	\RRR^{-}(t;\bfw)=O_{\bfw}\,\big(\, t^{(p-1)(q-1)+\epsilon}\,\big) 
\end{equation*}
for any $\epsilon>0$,
since the exponent $\frac{\kappa}{1+\kappa} (d-1)=(p-1)(q-1)+\epsilon_1$ for any 
$\epsilon_1>0$ (with $\epsilon_1=0$ if $q=2$).
On the other hand,  from Lemma~4.2\textit{(ii)} we immediately derive the following 
asymptotic formula: 
$$\tau(t_{\bfm})=v\, \,t_{\bfm}^{\,(p-1)(q-1)} 
\big(\,1+O(t_{\bfm}^{-1})\,\big),\quad 
t_{\bfm}\to\infty\,, $$
for some constant $v>0$. This formula together with Lemma~4.1 implies the lower 
bound  
\begin{equation*}
	\RRR^{-}(t_{\bfm};\bfw)\ge (v/2)\, \,t_{\bfm}^{\,(p-1)(q-1)} 
	\big(\,1+O(t_{\bfm}^{-1})\,\big),\quad 
	t_{\bfm}\to\infty\,.
\end{equation*}
A comparison of  the upper and lower bounds shows 
that the exponent $(p-1)(q-1)$ cannot 
be improved. Such examples with $\kappa\ge1$ can be constructed  
in all dimensions $d\ge3$.


\section{Simplices and their Fourier transform}
\label{sec5}

In this section, we compute  the Fourier expansion 
\begin{equation}
	N(t;\bfw, \bfu)\thicksim\sum\nolimits_{\bfm\in\Zz^d}\,
	N_{\bfm}(t;\bfw)\,e^{\,2i\pi\bfm\centerdot\bfu}, 
	\label{eq3.0}
\end{equation}
for $N(t;\bfw,\bfu)$ as a periodic function of $\bfu\in\Rr^d$. 
The result is stated below in Proposition~5.2.
Here and in what follows, the $\sim$ sign means that a Fourier series 
is considered simply as the list of its Fourier harmonics, ignoring all 
questions of convergence. 

Let $\chi(\EEE,\bfx)$ for $\bfx\in\Rr^d$ denote the indicator function of a 
subset $\EEE\subseteq\Rr^d$.
The number of integer points in the shifted simplex \eqref{eq1.22} can be 
written as
\begin{equation}
	N(t;\bfw,\bfu)=\sum\nolimits_{\bfm\in\Zz^d}\,\chi\, (\Delta^{-} 
	(t;\bfw), \bfm -\bfu), 
	\label{eq3.1}
\end{equation}
Clearly,   $N(t;\bfw,\bfu)$ as a function of $\bfu$ belongs to 
$L^{\infty}(\Rr^d/\Zz^d)$.

For a function $f\in L^1(\Rr^d)$, the sum $\sum\nolimits_{\bfm\in\Zz^d} f(\bfm 
-\bfu)$ is a well-defined periodic function belonging to  $L^1(\Rr^d/\Zz^d)$, and 
the Poisson summation formula gives its Fourier series
\begin{equation*}
\sum\nolimits_{\bfm\in\Zz^d} f(\bfm 
+\bfu)\,\thicksim\,\sum\nolimits_{\bfm\in\Zz^d}
\widetilde f(\bfm)\,e^{2i\pi\bfm\centerdot\bfu}\, ,	
\end{equation*}
where 
$\widetilde 
f(\bfy)=\int\nolimits_{\Rr^d}e^{-2i\pi\bfy\centerdot\bfx}\,f(\bfx)\,d\bfx$ 
is the Fourier transform  (see \cite[Chap.VII, Thm. 2.4]{27}).

Applying this formula to $f(\bfx)=\chi\, (\Delta^{-} 
(t;\bfw), \bfx -\bfu)$, we obtain the Fourier coefficients in \eqref{eq3.0}
\begin{align} 
	N_{\bfm}(t;\bfw)\,&=
	\widetilde
	\chi\, (\Delta^{-} 
	(t), -\bfm)=
	\int\nolimits_{\Delta^{-} (t;\bfw)}\,e^{\,2i\pi\bfm\centerdot\bfu}\, 
	d\bfu.
	\notag
	\\
	&=\,
	\frac{t^d}{w_1\dots w_d}\, 
	\,\XXX \left(\frac{m_1}{w_1}\,t,\dots, \frac{m_d}{w_d}\,t\right),
	\label{eq3.6}
\end{align}
where
\begin{equation}
	\XXX (\bfy)=
	\int\nolimits_{\varDelta}\,e^{\,2i\pi\bfy\centerdot\bfx}\, d\bfx\, 
	,\quad
	\bfy\in\Rr^d\, ,
	\label{eq3.7}
\end{equation} 
and $\varDelta =\{\bfx\in\Rr^d_{>0} : \, x_1+\cdots +x_d <1\}$ 
is the standard simplex in $\Rr^d$. Note that 
$\XXX (0)=\text{vol}\,\varDelta=(d\, !)^{-1} $.

The integral \eqref{eq3.7} is a smooth (indeed holomorphic) function
with a very specific asymptotic behavior: 
outside  the hyperplanes 
$\{\bfy: y_j=0\},\, j\in [d],$ and $\{\bfy : y_j= y_l\}, j\ne l$, 
the integral decays as 
$\bfy\to\infty$, while on the hyperplanes it has polynomial growth.
\\

\textbf{Definition 5.1.}
A point $\bfy\in\Rr^d$ is called 
\textit {nondiagonal} if it lies outside  the hyperplanes 
$\{\bfy: y_j=y_l\},\,j\ne l$ for all $j,l$ in $[d]$.

\textbf{Remark.}
Clearly,  points 
$$\Big(\,\frac{m_1}{w_1},\dots, 
\frac{m_d}{w_d}\,\Big)\, t,\quad \bfm 
=(m_1,\dots,m_d)\in\Zz^d\setminus \{\bf0\},$$ 
are nondiagonal if the quotients $\theta_{j,l}=\frac{w_l}{w_j}$ for $j\ne l$ 
are irrational.  
Therefore, in the case of irrational inclines $\Theta_{\widehat\bfw}$,  
we do not need to investigate the integral \eqref{eq3.7}
at ``diagonal points"  on the hyperplanes $\{\bfy :y_j=y_l\}$ for $j\ne l$,
since such points do not contribute to the Fourier coefficients \eqref{eq3.6}.
\\

The following notation will be used in our calculations. 
For a subset $J\subseteq [d]$, we write
$J'=[d]\setminus J$ and $|J|=\#\,J$. We set
\begin{equation*}
	\left.
\begin{aligned}
	&\Rr^J=\{\bfy=(y_1, \dots, y_d)\in\Rr^d: y_j\in\Rr\,\,\text{if}\,\, 
	j\in J 
	\,\,\text{and}\,\, y_j=0\,\, \text{if}\, \, j\in J'\},
	\\
	&\Rr^J_{\ne 0}=\{\bfy=(y_1, \dots, y_d)\in\Rr^d: y_j\ne0\,\,\text{if}\,\, 
	j\in J 
	\,\,\text{and}\,\, y_j=0\,\, \text{if}\, \, j\in J'\},
	\\
	&\Rr^J_{> 0}=\{\bfy=(y_1, \dots, y_d)\in\Rr^d: y_j>0\,\,\text{if}\,\, 
	j\in J 
	\,\,\text{and}\,\, y_j=0\,\, \text{if}\, \, j\in J'\}.
     \end{aligned}
\quad\right\}	
\end{equation*}
In particular, $\Rr^{\emptyset}_{\ne0}=\Rr^{\emptyset}_{>0}=\{\bf0\}$, 
$\Rr^{[d]}=\Rr^d,\,\,\Rr^{[d]}_{\ne0}=\Rr^d_{\ne0},\,\,\Rr^{[d]}_{>0}=\Rr^d_{>0}$.
The subsets $\Rr_{\ne0}^J$ and $\Rr_{>0}^J$ for $\,J\subseteq [d]$ form  
partitions of 
$\Rr^d$ and 
$\Rr_{\ge0}^d$:
\begin{equation*}
	\Rr^d=\bigsqcup\nolimits_{J\subseteq [d]}	\Rr^J_{\ne0}, \quad 
\Rr_{\ge0}^d=\bigsqcup\nolimits_{J\subseteq [d]}	\Rr^J_{>0}\, .	
\end{equation*}
Hence,
$1=\sum\nolimits_{J\subseteq [d]}\,\chi (\Rr_{\ne0}^J,\,\bfy)$ and
$\chi (\Rr_{\ge0}^d,\bfy)=\sum\nolimits_{J\subseteq [d]}\,\chi 
(\Rr_{>0}^J,\,\bfy).$

Let us put   
\begin{align}
	\Zz^J=\Zz^d\bigcap\Rr^J,\quad
	\Zz^J_{\ne0}=\Zz^d\bigcap\Rr_{\ne0}^J\, ,\quad
	\Zz^J_{>0}=\Zz^d\bigcap\Rr_{>0}^J\, ,
	\label{eq2.02}
\end{align}
In particular, $\Zz^{\emptyset}_{\ne0}=\Zz^{\emptyset}_{>0}=\{0\},\,\, 
\Zz^{[d]}=\Zz^d,\,\,
\Zz^{[d]}_{\ne0}=\Zz^d_{\ne0}$, and $\Zz^{[d]}_{>0}=\Zz^d_{>0}$.
We have the partitions
\begin{equation}
	\Zz^d=\bigsqcup\nolimits_{J\subseteq [d]}	\Zz^J_{\ne0}, \quad 
	\Zz_{\ge0}^d=\bigsqcup\nolimits_{J\subseteq [d]}	\Zz^J_{>0}\, .	
	\label{eq3.10*}
\end{equation}
Hence,
\begin{align*}
\chi 
(\Zz^d,\bfy)=\sum\nolimits_{J\subseteq [d]}\,\chi (\Zz_{\ne0}^J,\,\bfy),\quad
\chi 
(\Zz_{\ge0}^d,\bfy)=\sum\nolimits_{J\subseteq [d]}\,\chi (\Zz_{>0}^J,\,\bfy).
\end{align*}
Define also subsets  $Z_j\subset\Zz^d$ for $j\in[d]$ by 
\begin{align*}
	Z_j=&\bigcup\nolimits_{J\supseteq\{j\}}\, \Zz_{\ne0}^J\notag
	\\
	=&\{\bfm=(m_1,\dots,m_d)\in\Zz^d: m_j\in\Zz_{\ne 0}\,\, \text{and}\,\,
	m_l\in\Zz\,\,\, \text{if}\,\,\, l\ne j\}.
\end{align*}
Hence,
\begin{equation}
\chi (Z_j,\bfy)=\sum\nolimits_{J\supseteq\{j\}}\, \chi (\Zz_{\ne0}^J, \bfy)\, 
,	
\label{eq2.003}
	\end{equation}
since the subsets $\Zz_{\ne0}^J$ are disjoint for different $J\subseteq [d]$.	

\begin{proposition}
	Let $\bfy\in\Rr^J_{\ne0}$ for $J\ne\emptyset$ be a nondiagonal point; then
	\begin{equation}
		\XXX (\bfy)=X_1(\bfy)+X_2(\bfy)\, ,
		\label{eq3.29}
	\end{equation}
	where	
	\begin{align*}
		X_1(\bfy)=&(-1)^{|J|}
		\sum_{n=|J|}^d\left(\frac{1}{2i\pi}\right)^n
		\frac{1}{(d-n)!} 
		\sum_{\bfn\in\Zz_{>0}^J: |\bfn|_1=n}
		\prod\nolimits_{j\in J}\,\frac{1}{y_j^{n_j}}\, ,
	\\
	X_2(\bfy)=	&\left(\frac{1}{2i\pi}\right)^d\,
		\sum_{j\in J}\,\,
		\frac{e^{2i\pi y_j}}{ y_j\prod\nolimits_{l\in [d], l\ne j}(y_j -y_l)}
		\, .
	\end{align*}
	Here $|\bfn|_1=\sum\nolimits_{j\in J}\,n_j$ for 
	$\bfn\in\Zz_{>0}^J$.
\end{proposition}
Before proceeding to the proof of Proposition~5.1 below, we need  the following 
auxiliary results:
	\\
	
\textit{(i)}
For a subset $J\subset [d],\, J\ne\emptyset,$ and its complement $J'$, we write 
$\bfx=\bfx_J 
+\bfx_{J'}$, where
$\bfx_J\in\Rr^J,\, \bfx_{J'}\in\Rr^{J'}$, and 
$\varDelta_J=\{\bfx\in\Rr^J_{>0} :\,\sum\nolimits_{j\in J}x_j 
<1\}$. 	With this notation,  the integral \eqref{eq3.7} 
can be written as
\begin{equation}
	\XXX (\bfy)=
	\int\nolimits_{\varDelta_{J'}}\,d\bfx_{J'}\,
	e^{\, 2i\pi\,\bfy_{J'}\centerdot\bfx_{J'}}\,\beta(\bfx_{J'})^{|J|}\,\,
	 \XXX_J(\,\beta(\,\bfx_{J'})\,\bfy_J\,)\, ,
	\label{eq3.12}
\end{equation}
where $\beta(\bfx_{J'})=1-\bfe_1\centerdot\bfx_{J'}=1-\sum\nolimits_{j\in J'}\, x_j$
and
\begin{equation}
	\XXX_J(\bfy_J)=
	\int\nolimits_{\varDelta_J}\,e^{\,2i\pi\bfy_J\centerdot\bfx_J} 
	d\bfx_J\, ,
	\label{eq3.13}
\end{equation}

\textit{(ii)} 
For a subset $I\subseteq [d],\, I\ne\emptyset,$ and an integrable 
function $f(t)$ for $t\in [0,1]$, we have  
\textit{Liouville's formula}
\begin{equation}
	\int\nolimits_{\varDelta_I}\,f(\bfe_1\centerdot\bfx_I)\,d\bfx_I\,=
	\int\nolimits_0^1\,f(t)\, d v_I(t)\, =\,
	\frac{1}{(|I|-1)!}\int\nolimits_0^1\,f(t)\, t^{|I|-1}\, d t\, ,
	\label{eq3.14}
\end{equation}
where $v_I(t)=(|I|!)^{-1}\, t^{|I|}$ is the volume of the simplex 
$\varDelta_I(t)=\{\bfx\in\Rr^I_{>0} :\,\sum\nolimits_{j\in I}x_j 
<t\}$.
\\

\textit{(iii)} 
The following useful formula for the integral \eqref{eq3.13} 
for special  $\bfy$ was given by Borda \cite[Equality (18)]{5}. Here the formula 
is written in our notation, and for the convenience of the reader, its proof is 
also provided.
\begin{lemma}
	Let $\bfy_J\in\Rr^J_{\ne0}$ be a nondiagonal point. Then
	\begin{equation}
		\XXX_J(\bfy_J)=
	\left(\frac{1}{2i\pi}\right)^{|J|}\,	
	\sum_{j\in J}\,\frac{e^{\,2i\pi y_j} - 
	1}{y_j\,\,\prod\nolimits_{l\in J, l\ne j}(y_j - y_l)} \, ,
		\label{eq3.15}
	\end{equation}
	\end{lemma}
	\begin{proof} Without loss of generality, we assume that $J=[n]$ and 
	prove \eqref{eq3.15} by induction on $n$. For $n=1$, the formula is 
	true:
	\begin{equation*}
		\XXX_{[1]}(\bfy_1)=\int\nolimits_0^1 e^{2i\pi y_1 x_1}\, d x_1 =
		\,\frac{e^{\,2i\pi y_j} - 1}{2i\pi\, y_1} \, .
	\end{equation*} 
Suppose that \eqref{eq3.15}	holds for the subset $[n-1]\subset [n]$. Then 
from \eqref{eq3.12},  after some calculation, we obtain
\begin{align}
&\XXX_{[n]}(\bfy_{[n]})=
\left(\frac{1}{2i\pi}\right)^{n-1}\int\nolimits_0^1\, dx_n\, e^{2i\pi y_n x_n}	
\sum_{j=1}^{n-1}\,\frac{e^{\,2i\pi y_j (1-x_n)} - 1}
{y_j\,\prod\nolimits_{k\in [n-1],k\ne j}(y_j - y_k)}\notag
	\\
&=\left(\frac{1}{2i\pi}\right)^n \left(\sum_{j=1}^{n-1}\,\frac{e^{\,2i\pi y_j } 
- 1}{y_j\,\prod\nolimits_{k\ne j}(y_j - y_k)} -
	\sum_{j=1}^{n-1}\,
	\frac{e^{\,2i\pi y_n}-1}{y_n\prod\nolimits_{k\ne j}(y_j - y_k)} 
	\right)	 .
	\label{eq3.17}
\end{align}

Consider the partial fraction decomposition 
\begin{equation}
	\frac{1}{\prod\nolimits_{k=1}^{n-1}(z - y_k)}=\sum_{j=1}^{n-1}\,
	\frac{1}{z-y_j}\,\frac{1}{\prod\nolimits_{k\ne j,n} (y_j-y_k)}, 
	\quad z\in \Cc .
	\label{eq3.18}
\end{equation}
Putting here $z=y_n$, we obtain
\begin{align}
	\frac{1}{\prod\nolimits_{k=1}^{n-1}(y_n - y_k)}=&-\sum_{j=1}^{n-1}\,
	\frac{1}{(y_j-y_n)\prod\nolimits_{k\ne j,n} (y_j-y_k)}\notag
	\\ 
	=&-\sum_{j=1}^{n-1}\,\frac{1}{\prod\nolimits_{k\ne j} (y_j-y_k)}\, .
	\label{eq3.19}
\end{align}
Substituting \eqref{eq3.19} into the second sum in \eqref{eq3.17} completes 
the proof. 
\end{proof}

\textit{(iv)} 
The following simple observation plays a key role in the subsequent 
calculations.
\begin{lemma}[\bf{The Symmetrization Lemma}]
	Let $J\ne\emptyset$ and let $\bfy\in\Rr^J_{\ne0}$ be a nondiagonal point. 
	Then for all integers $n\ge |J|$, we have 
	\begin{equation}	
		\sum_{j\in J}\,\frac{1}{y_j^{n-|J|+1}\,\prod\nolimits_{k\in J,k\ne 
		j}(y_j - 
		y_k)} =(-1)^{|J|-1}\sum_{\bfn\in\Zz_{>0}^J: |\bfn|_1=n}
		\prod\nolimits_{j\in J}\frac{1}{y_j^{n_j}}\, ,
		\label{eq3.20}
	\end{equation}
	here $|\bfn|_1=\sum\nolimits_{j\in J}\,n_j$.
	\end{lemma}
\begin{proof} Similarly to \eqref{eq3.18}, we have the partial fraction 
decomposition
\begin{equation*}
	\frac{1}{\prod\nolimits_{j\in J}(z - y_j)}=\sum\nolimits_{j\in J}\,
	\frac{1}{z-y_j}\,\frac{1}{\prod\nolimits_{k\in J,\,k\ne j} (y_j-y_k)}, 
	\quad z\in \Cc .
\end{equation*}	
	Rewrite this identity in the form
\begin{equation*}
	\prod_{j\in J}y_j^{-1}(1 - \frac{z}{y_j})^{-1}=(-1)^{|J|-1}\sum_{j\in J}\,
	(1 - \frac{z}{y_j})^{-1}\,\,\frac{1}{y_j\prod\nolimits_{k\in J,k\ne 
	j}(y_j-y_k)} . 
\end{equation*}

Substituting the expansion 
$(1 - \frac{z}{y_j})^{-1} =\sum\nolimits_{l_j\ge0}\,\,(\frac{z}{y_j})^{l_j}$
into this formula and equating terms with the same power of $z$, we obtain 
\begin{equation}	
	\sum_{j\in J}\,\frac{1}{y_j^{l+1}\,\prod\nolimits_{k\in J,k\ne 
			j}(y_j - 
		y_k)} =(-1)^{|J|-1}\sum_{|\bfl|_1=l}
	\prod\nolimits_{j\in J}\frac{1}{y_j^{l_j+1}}\, ,
	\label{eq3.20*} 
\end{equation}
where $|\bfl|_1=\sum_{j\in J}l_j$. Let $\bfn=(n_1,\dots ,n_d)\in\Zz_{>0}^d$ be 
a point with $n_j=l_j +1$ for $j\in J$ and $n_j=0$ for $j\in J'$. Then 
$|\bfn|_1=|\bfl|_1+|J|=l+|J|$. Replacing $\bfl$ by $\bfn$ in \eqref{eq3.20*}, 
we obtain \eqref{eq3.20}.
\end{proof}

\begin{proof}[\bf{Proof of Proposition~5.1.}] 
	For $J=[d]$, Lemmas~5.1 and 5.2 (with $n=d$) imply
	\begin{equation}
		\XXX(\bfy)=\left(\frac{-1}{2i\pi}\right)^d
		\frac{1}{\prod\nolimits_{j\in [d]} y_j}	 +
		\left(\frac{1}{2i\pi}\right)^d	
		\sum\nolimits_{j=1}^d\,\frac{e^{\,2i\pi y_j}}{y_j\,
			\prod\nolimits_{k\ne j}(y_j - y_k)} \, .
		\label{eq3.23}
		\end{equation}
		This proves \eqref{eq3.15} for $J=[d]$.

		Here it is worth noting the following.
		Let $J\subset [d], \,J\ne\emptyset$. 
		Since $\XXX(\bfy)$ is a smooth function, one could 
		consider the limit of $\XXX(\bfy)$ as $y_j\to 0$ for $j\in J'$  
		to compute $\XXX_J(\bfy_J)$. This approach is possible, but it leads 
		to rather lengthy calculations to resolve the singularities in 
		\eqref{eq3.23}. We proceed differently.
		
		Substituting \eqref{eq3.15} into \eqref{eq3.12}, we find
		\begin{equation}
			\XXX(\bfy_J)=
			\left(\frac{1}{2i\pi}\right)^{|J|-1}\,	
			\sum_{j\in J}\,\frac{1}
				{\prod\nolimits_{l\in J, l\ne j}(y_j - y_l)} \,\, \III(y_j)\, ,
			\label{eq3.24}
		\end{equation}
		where
		\begin{equation*}
			\III (y)=
			\int\nolimits_{\varDelta_{J'}}\,d\bfx_{J'}\,\,
			\frac{e^{2i\pi\, \beta(\bfx_{J'})\,y}-1}{2i\pi y}\, ,\quad y\ne0\, ,
		\end{equation*}
	Using Liouville's formula \eqref{eq3.14} with $I=J'$, we find
	\begin{equation*}
		\III (y)=
		\frac{1}{(|J'|-1)!}\int\nolimits_0^1\,dt\,\,t^{|J'|-1}\,\,\,
		\frac{e^{2i\pi (1-t)y}-1}{2i\pi y}\, ,
	\end{equation*}
	This integral  can be computed as follows:
	\begin{align*}
		\III (y)=&
		\frac{1}{(|J'|-1)!}\int\nolimits_0^1\,dt\,\,t^{|J'|-1}\,\,\,
	\int\nolimits_0^{1-t} e^{2i\pi yu} du \notag
	\\
	=& -\frac{1}{|J'|!}\int\nolimits_0^1\,dt\,\,t^{|J'|}\,e^{2i\pi y(1-t)}\notag
	\\
	=&-\sum\nolimits_{l=0}^{|J'|}\frac{1}{(|J'|-l)!\,(2i\pi y)^{l+1}} 
	+\frac{e^{2i\pi 
			y}}{(2i\pi y)^{|J'|+1}}\, ,
	\end{align*}	
where the last integral can be computed by 
integration by parts or found in tables.		

Setting $|J'|=d-|J|$ and $n=|J|+l$,  we write
	\begin{align*}
		\III (y)=-\sum\nolimits_{n=|J|}^d\frac{1}{(d-n)!\,(2i\pi 
		y)^{n-|J|+1}} 
		+\frac{e^{2i\pi y}}{(2i\pi y)^{d-|J|+1}}\, .
	\end{align*}
	Substituting this formula  into \eqref{eq3.24}, we obtain
	\begin{equation*}
		\XXX (\bfy)=X_1(\bfy)+X_2(\bfy)\, ,
	\end{equation*}
	where	
	\begin{align}
		X_1(\bfy)=
		-\sum_{n=|J|}^d\left(\frac{1}{2i\pi}\right)^n 
		\sum_{j\in J}
		\frac{1}{(d-n)!\, 
			y_j^{n-|J|+1}\prod\nolimits_{l\in J, l\ne j}(y_j - y_l)} 
		\label{eq3.30}
	\end{align}
	and
	\begin{align}
		X_2(\bfy)=\left(\frac{1}{2i\pi}\right)^d\,
		\sum_{j\in J}\,\,
		\frac{e^{2i\pi y_j}}{ y_j^{d-|J|+1}\prod\nolimits_{l\in J, l\ne 
		j}(y_j - y_l)}\, .
		\label{eq3.31}
	\end{align}
	
	Now we apply Lemma~5.2  to  symmetrize the inner sum in \eqref{eq3.30}. 
This gives
	\begin{align*}
		X_1(\bfy)=(-1)^{|J|}
		\sum_{n=|J|}^d\left(\frac{1}{2i\pi}\right)^n
		\frac{1}{(d-n)!} 
		\sum_{\bfn\in\Zz_{>0}^J: |\bfn|_1=n}
		\prod\nolimits_{j\in J}\frac{1}{y_j^{n_j}}\, .
	\end{align*}
	The sum \eqref{eq3.31} can be written as
	\begin{align*}
		X_2(\bfy)=\left(\frac{1}{2i\pi}\right)^d\,
		\sum_{j\in J}\,\,
		\frac{e^{2i\pi y_j}}{ y_j\prod\nolimits_{l\in [d], l\ne j}(y_j -y_l)}\,,
	\end{align*}
because $\bfy\in\Rr^J_{\ne 0}$ and  $y_l=0$ for $l\in J'$.
	\end{proof}

We are now in a position to state the main result of this section.
\begin{proposition} 
	Let the inclines $\theta_{j,l}=\frac{w_l}{w_j}\in\Theta_{\widehat\bfw}$ for 
	$l\ne j$  be irrational. Then
	$N(t;\bfw,\bfu)$  has the Fourier expansion
	\begin{equation*}
		N(t;\bfw,\bfu)\thicksim\sum\nolimits_{\bfm\in\Zz^d}\,
		N_{\bfm}(t;\bfw)\,e^{\,2i\pi\bfm\centerdot\bfu}, 
	\end{equation*}
	with Fourier coefficients
	\begin{equation*}
		N_{\bfm}(t;\bfw)=\frac{1}{d!\,w_1\dots w_d} \,\,
		Q_{\bfm}(t;\bfw)+R_{\bfm}(t;\bfw)\,, 
	\end{equation*}
	
	The coefficients $Q_{\bfm}(t;\bfw)$ are given by
	\begin{equation}
		Q_{\bfm}(t;\bfw)=\,t^d +\,
		\sum\nolimits_{J\subseteq 
		[d],\,J\ne\emptyset}\,\,\chi(\Zz_{\ne0}^J,\bfm)\,\, 
		b_{\bfm}(t;\bfw, J)\, ,
		\label{eq2.09}
	\end{equation}
	where 
	\begin{equation}
		b_{\bfm}(t;\bfw, J)
		=\sum\nolimits_{n=|J|}^d \, 
		b_{n, \bfm}(\bfw,J)\,\,\frac{d!}{n!(d-n)!}\,\,t^{d-n}\, 
		\label{eq2.010}
	\end{equation}
	with
	\begin{equation}
		b_{n, \bfm}(\bfw,J)=(-1)^{|J|}\,n!\,\left(\frac{1}{2i\pi}\right)^n\,
		\sum\nolimits_{\bfn\in\Zz^J_{>0}:|\bfn|_1=n}\, 
		\prod\nolimits_{j\in 
			J}\frac{w_j^{n_j}}{m_j^{n_j}}\, .
		\label{eq3.37}
	\end{equation}
	
	The coefficients $R_{\bfm}(t;\bfw)$ are given by
	\begin{equation}
		R_{\bfm}(t;\bfw)=\left(\frac{1}{2i\pi}\right)^d\,
		\sum_{j=1}^d \,\, \chi(Z_j,\bfm)\,\,
		\,\frac{e^{2i\pi\, \frac{m_j}{w_j}\, t}}
		{m_j\prod\nolimits_{l\in [d], l\ne j}(\theta_{j,l}\,m_j -m_l)}\,.
		\label{eq3.38}
	\end{equation}
\end{proposition}
\begin{proof}
	Using the partition \eqref{eq3.10*}, we rewrite \eqref{eq3.0} as follows
	\begin{equation}
N(t;\bfw, \bfu)\thicksim\sum\nolimits_{J\subseteq [d]}
		\sum\nolimits_{\bfm\in\Zz^d}\,
		\chi (\,\Zz_{\ne0}^J,\bfm\,)\,
		N_{\bfm}(t;\bfw)\,e^{\,2i\pi\bfm\centerdot\bfu}, 
		\label{eq2.006}
	\end{equation}
Taking \eqref{eq3.6} and \eqref{eq2.003} into account, we apply Proposition~5.1   
to  \eqref{eq2.006}.
The first term  $X_1$ in \eqref{eq3.29}  immediately leads to  
\eqref{eq2.09}-\eqref{eq3.37}, and the second term $X_2$ leads to \eqref{eq3.38}.
	\end{proof}


\section{Fourier expansions for the error terms}
\label{sec6}

We first derive the Fourier expansion for the periodization of the multiple 
Bernoulli polynomials defined in \eqref{eq2.22a}.
We consider
\begin{align*}
	\BBB_k(t;\bfw, \{\bfu\})
	=\sum\nolimits_{n=0}^k\,\frac{k!}{n!(k-n)!}\, 
	\BBB_n(\bfw,\{\bfu\})\, t^{k-n}
\end{align*}
with the periodic coefficients 
\begin{align}
	&\BBB_n(\bfw,\{\bfu\})\notag
	\\
	&=\sum_{n_1+\dots +n_d=n}\,\frac{n!}{n_1!\dots n_d!}\, 
	B_{n_1}(\{u_1\})\dots B_{n_d}(\{u_d\})\, w_1^{n_1}\dots w_d^{n_d}.
	\label{eq2.24}
\end{align}

It is useful to keep in mind that $B_0(\{u\})=1,\,B_1(\{u\})=\{u\}-1/2=((u))$
(the sawtooth function), 
$B_n(\{u\})$ for $n\ge 2$ 
are continuous for $u\in\Rr$ and all
$B_n(\{u\})$ for $n\ge 1$ are $\CCC^{\infty}$ functions
outside the neighborhoods of integer points. Hence, $\BBB_k(t;\bfw, \{\bfu\})$ are
$\CCC^{\infty}$ functions of $t\in\Rr_{>0}$  and $\bfu\in\UUU(\delta)$
with jumps on the hyperplanes 
$\{\bfu: u_j=l\}$ for $l\in\Zz$ and $j\in[d]$.

The following Fourier expansions  
\begin{equation}
	B_n(\{u\})\,\sim\,- \frac{n!}{(2i\pi)^n}
	\sum\nolimits_{m\in\Zz_{\ne 0}}\,\frac{e^{2i\pi mu}}{m^n}\,,\quad 
	n\in\Zz_{>0},
	\label{eq2.25}
\end{equation}
are well known (see \cite[Sec.9.1.4]{9}). 
To substitute \eqref{eq2.25} into \eqref{eq2.24}, 
 the formula \eqref{eq2.24} should be written slightly differently. 
With the notation \eqref{eq2.02}, we rewrite \eqref{eq2.24} as follows: 
\begin{align}
	\BBB_n(\bfw,\{\bfu\})
	=n!\,\sum\nolimits_{J\subseteq 
		[d]:|J|\le n}
	\sum\nolimits_{\bfn\in\Zz^J_{>0}:|\bfn|_1=n}\,\prod\nolimits_{j\in 
		J}\frac{B_{n_j}(\{u_j\})\,w_j^{n_j}}{n_j!}\, .
	\label{eq2.30}
\end{align}
In this formula, the factors $B_{n}(\{u\})$ with $n=0$ and $n\ge1$ are separated.
Note that $|\bfn|_1\ge |J|$ for all 
$\bfn\in\Zz^J_{>0}$.
Substituting \eqref{eq2.25} into \eqref{eq2.30},  
after some calculations, we arrive at the following.
\begin{proposition}
	For $k\ge1$, we have the Fourier expansions
	\begin{align*}
		\BBB_k(t;\bfw,\{\bfu\})
		\sim\,
		\sum\nolimits_{\bfm\in\Zz^d} 
		b_{k, \bfm}(t;\bfw)\,e^{2i\pi \bfm\centerdot\bfu}\, 
	\end{align*}
	with Fourier coefficients 
	\begin{align}
		b_{k, \bfm}(t;\bfw)= t^k +
		\sum\nolimits_{J\subseteq [d],0<|J|\le 
			k}\chi (\Zz_{\ne0}^J,\bfm) \,\,
		b_{k, \bfm}(t;\bfw, J)\, ,
		\label{eq2.33*}
	\end{align}
	where
	\begin{equation*}
		b_{k, \bfm}(t;\bfw, J)
		=\sum\nolimits_{n=|J|}^k \, 
		b_{n, \bfm}(\bfw,J)\,\,\frac{k!}{n!(k-n)!}\,\,t^{k-n}\, ,
	\end{equation*}
	and
	\begin{equation*}
		b_{n, \bfm}(\bfw,J)=(-1)^{|J|}\,n!\,\left(\frac{1}{2i\pi}\right)^n\,
		\sum\nolimits_{\bfn\in\Zz^J_{>0}:|\bfn|_1=n}\, 
		\prod\nolimits_{j\in 
			J}\frac{w_j^{n_j}}{m_j^{n_j}}\, ,
	\end{equation*}
	
	For $k\ge d$, the restriction $|J|\le k$ in 
	\eqref{eq2.33*} is unnecessary, and \eqref{eq2.33*} takes the form
	\begin{align*}
		b_{k, \bfm}(t;\bfw)= t^k +
		\sum\nolimits_{J\subseteq [d],J\ne\emptyset}\chi (\Zz_{\ne0}^J,\bfm) \,\,
		b_{k, \bfm}(t;\bfw, J)\, .
	\end{align*}
\end{proposition}


Using Propositions 5.2 and 6.1, we immediately obtain the Fourier expansion of the 
error term
\begin{equation*}
	R(t;\bfw,\{\bfu\})=
	N(t;\bfw,\bfu)-
	\frac{1}{d!\,w_1\dots w_d} \,\,
	\BBB_d(t;\,\bfw, \{\bfu\})\,. 
\end{equation*}
\begin{proposition} 
	Let the inclines $\theta_{j,l}=\frac{\widehat w_l}{\widehat 
	w_j}\in\Theta_{\widehat\bfw}$ for $l\ne j$  
	be irrational. 
	Then
	$R(t;\bfw,\{\bfu\})$  has the Fourier expansion
\begin{equation*}
	R(t;\bfw,\{\bfu\})\sim\,\sum\nolimits_{\bfm\in\Zz^d}\, R_{\bfm}(t;\bfw)\,\, 
	e^{2i\pi \bfm\centerdot\bfu}\, 
	\label{eq4.2}
\end{equation*}
with Fourier coefficients 
\begin{equation*}
R_{\bfm}(t;\bfw)=\sum\nolimits_{j=1}^d \, R_{j,\bfm}(t;\bfw)\, ,
	\end{equation*}
	where
\begin{equation*}
	R_{j,\bfm}(t;\bfw)=\left(\frac{1}{2i\pi}\right)^d\,
	\, \chi(Z_j,\bfm)\,
	\frac{e^{2i\pi \frac{m_j}{w_j}t}}
	{m_j\prod\nolimits_{l\in [d], l\ne j}(\theta_{j,l}\,m_j -m_l)}\, ,
\end{equation*}
\end{proposition}

It is worth noting the following feature of
the Fourier expansion described in Proposition 6.2. 
If all questions of convergence are ignored, we may write
\begin{equation}
	R(t;\bfw,\{\bfu\})\,=\,\sum\nolimits_{j=1}^d\,\, R_{j}(t;\bfw,\{\bfu\})\, ,  
	\label{eq3.002}
\end{equation}
where
\begin{align}
	R_j(t;\bfw, &\,\{\bfu\})\notag
	\\
	&\sim\,\left(\frac{1}{2i\pi}\right)^d
 \,\sum_{m_j\ne0} 
	\frac{e^{2i\pi \frac{m_j}{w_j}t+2i\pi m_ju_j}}{m_j}
	\prod\nolimits_{l\in [d], l\ne j}
	\sum_{m_l\in\Zz}
	\frac{e^{2i\pi m_lu_l}}{\theta_{j,l}m_j -m_l}\, .
	\label{eq3.06a}
\end{align}

This formal calculation can be continued as follows:
Using in \eqref{eq3.06a} the Fourier expansion 
\begin{equation}
	\frac{e^{2i\pi a\, ((u))}}{2i\sin \pi a}\, =\,	
	\frac{e^{2i\pi a \{u\}}}{e^{2i\pi 
			a}-1}\,\sim\,\sum_{m\in\Zz}\frac{e^{2i\pi 
			mu}}{2i\pi 
		(a-m)},\quad a\notin\Zz\, .
	\label{eq4.07}
\end{equation}
with $a=\theta_{j,l}\,m_j$ and $m=m_l$, we obtain
\begin{align}
	R_j(t;\bfw, \bfu)&\sim\,\,\frac{1}{2i\pi}\,
	\,\sum_{m\in\Zz_{\ne 0}} \,\, 
	\frac{e^{2i\pi \frac{m}{w_j}\left(t+\bfw\centerdot\{\bfu\}\right)}}
	{m\prod\nolimits_{l\in [d], l\ne j}(e^{2i\pi\,\theta_{j,l}m}-1)}\notag
	\\
	&\sim\,\frac{1}{2^{d-1}\pi}\,\,\sum_{m=1}^{\infty} 
	\frac{\cos\left(\frac{2i\pi m}{w_j}
		\big(t+\bfw\centerdot((\bfu))\,\big)+\frac{\pi 
			d}{2}\right)}{m\prod\nolimits_{l\in 
			[d], l\ne j}\sin(\pi\,\theta_{l,j}m)}\, .
	\label{eq4.08*}
\end{align}

We do not discuss this interesting formal expansion in more detail here. 
We note only that this expansion can be easily interpreted in terms of 
distributions 
(generalized functions). 

Let us draw attention to the following circumstance.
The error term $R(t;\bfw,\{\bfu\})$, as a function of $\bfu$, is periodic and 
integrable on the torus $\Rr^d/\Zz^d$. Thus, its Fourier coefficients satisfy
$R_{\bfm}(t;\bfw)\to 0$ as $\bfm\to\infty$ by the Riemann--Lebesgue lemma. 
At the same time, the coefficients 
$R_{j,\bfm}(t;\bfw)$  grow because of the small denominators in \eqref{eq3.06a}.
This means that the expansions \eqref{eq3.06a} cannot be Fourier series of  
integrable functions. 
Therefore, the equality \eqref{eq3.002} should be understood as a representation of 
the integrable function  $R(t;\bfw,\bfu)$  as a sum of distributions 
$R_j(t;\bfw, \,\bfu)$ for $j\in [d]$.


\section{Smoothed Fourier series. Proof of the Main Theorem}
\label{sec7}

Recall that the space $\frak{S}(\Rr^d)$ of \textit{rapidly decreasing} 
functions consists of $\CCC^{\infty}$ functions $\varphi (\bfx)$ of $\bfx\in\Rr^d$
that satisfy
\begin{equation*}
	\big|\,\frac{\partial^{b_1+\cdots+b_d}}{\partial x^{b_1}
		\dots\partial x^{b_d}}\,\,\varphi(\bfx)\,\big| < C_{\varphi,\bfb,A}\, 
	(\,1+|\bfx|_{\infty})^{-A},
	\quad |\bfx|_{\infty}=\max\nolimits_{j\in[d]}\,|x_j|\,,
	\label{eq6.02**}	
\end{equation*}
for all multi-indexes $\bfb=(b_1,\dots,b_d)\in\Zz_{\ge0}^d$
and for arbitrarily large $A>0$  (see \cite[Chap.1, Sec. 3]{27}). 

\begin{lemma}
	Let a set of numbers 
	$\bf\theta = (\theta_1,\dots,\theta_{d-1})\in\Rr^{d-1}$  be
	$\kappa$-multiplicatively approximable and let
	$\varphi\in\frak{S}(\Rr^d)$.
	Then the series 
	\begin{align*}
		\sum_{m_d\ne0} 
		\frac{1}{|m_d|}
		\prod_{l\in [d-1]}\,
		\sum_{m_l\in\Zz}
		\frac{1}{|\,\theta_{l} m_d -m_l\,|}\,\, |\varphi (T^{-1}\bfm)| 
	\end{align*}	
converges.
\end{lemma}
\begin{proof}
By Definition~2.1,	
for a $\kappa$-multiplicatively approximable set, we have	
\begin{align}
	|m_d\prod\nolimits_{l\in [d], l\ne j}(\theta_{l}m_d -m_l)|
	\ge |m_d|\prod\nolimits_{l\in [d-1]}\,\, \langle 
	\theta_{l}\,\, m_d\rangle >c_{\kappa} \, |m_d|^{-\kappa}. 
	\label{eq4.14}
\end{align}
Therefore, the series under consideration is
$\ll\sum_{\bfm\in\Zz^d}|m_d|^{\kappa}\,(\,1+|\bfm|_{\infty})^{-A}<\infty$,
since $A$ can be chosen arbitrarily large.
\end{proof}

\label{sec4}
For $T>0$, we set
\begin{equation*}
	\omega_T(\bfx)=T^d\,\omega(T\,\bfx)\, ,\quad
	\omega (\bfx)=\prod\nolimits_{j=1}^d \omega (x_j)  \, , 
\end{equation*}
where $\omega (x)$ for $x\in\Rr$ is a nonnegative even $\CCC^{\infty}$ function 
supported on $[-\frac12,\frac12]$ that satisfies 
$\int\nolimits_{\Rr}\,\omega(x)\,dx=1$. The Fourier transform is
$\widetilde\omega_T(\bfy)=\widetilde\omega(T^{-1}\bfy)$. 
In the sense of distributions, $\omega_T(\cdot)$ 
tends to the Dirac delta function as $T\to\infty$. 

Consider the following convolutions:
\begin{equation}
	\left.
\begin{aligned}
R*\omega_T\,(t;\bfw,\bfu)=&
\int\nolimits_{\Rr^d}R(t;\bfw,\bfx)\,\,\omega_T(\bfu-\bfx)\,d\bfx\, , 	
\\
\BBB_d*\omega_T\,(t;\bfw,\bfu)=&
\int\nolimits_{\Rr^d}\BBB_d(t;\bfw,\{\bfx\})\,\,\omega_T(\bfu-\bfx)\,d\bfx\, ,	
\\
N*\omega_T\,(t;\bfw,\bfu)=&
\int\nolimits_{\Rr^d}N(t;\bfw,\bfx)\,\,\omega_T(\bfu-\bfx)\,d\bfx	
\end{aligned}
\quad\right\}
\label{eq4.09l}
\end{equation}
which are related by
\begin{equation*}
N*\omega_T\,(t;\bfw,\bfu)=\frac{1}{d!\,w_1\dots w_d} \,\,
\BBB_d*\omega_T\,(t;\bfw,\bfu)+R*\omega_T\,(t;\bfw,\bfu)\, .	
\end{equation*}

\begin{proposition}
Let the inclines $\Theta_{\widehat\bfw}$ be $\kappa$-multiplicatively 
approximable. Then we have the bounds
\begin{equation}
	\sup\nolimits_{\bfu\in\Rr^d}\,|R*\omega_T(t;\bfu)|\,\le\,
\,\frac{1}{2^{d-1}\pi}\,\,	\,\SSS_0(\,\Theta_{\widehat\bfw},\, \omega_T)\, ,
	\label{eq4.012}
\end{equation}
where
$\SSS_0(\,\Theta_{\widehat\bfw},\,\omega_T)=\sum\nolimits_{j=1}^d\,\,
\SSS_0(\,{\bf{\theta}}_j,\,\omega_T)$ and
\begin{align*}
	\SSS_0(\,{\bf{\theta}}_j, \,\omega_T)=
	\,\sum_{m=1}^{\infty} 
	\frac{|\widetilde\omega(T^{-1}\,m)|}{m\prod\nolimits_{l\in 
			[d], 
			l\ne j}|\sin(\pi\,\theta_{j,l}m)|}\, .
\end{align*}
\end{proposition}
\begin{proof}
 It follows from Proposition~6.2 that
\begin{equation*}
	R*\omega_T(t;\bfw,\bfu)\,=\,\sum\nolimits_{j=1}^d\,\, 
	R_{j}*\omega_T(t;\bfw,\bfu)\, ,  
\end{equation*}
where
\begin{align}
	R_j*\omega_T(t;\bfw, \bfu)\notag
	=\,\left(\frac{1}{2i\pi}\right)^d
	\,&\sum_{m_j\ne0} 
	\frac{e^{2i\pi 
	(\frac{t}{w_j}+u_j)m_j}\,\,\widetilde\omega(T^{-1}\,m_j)}{m_j}\notag
	\\
	&\times\prod\nolimits_{l\in [d], l\ne j}
	\sum_{m_l\in\Zz}
	\frac{e^{2i\pi m_lu_l}\,\,\widetilde\omega(T^{-1}\,m_l)}{\theta_{j,l}m_j 
	-m_l}\, .
	\label{eq4.06*}
\end{align}
Since $\widetilde\omega\in\frak{S}(\Rr^d)$, this series  
 converges absolutely by Lemma~7.1.
Using \eqref{eq4.07}, we find 
\begin{align*}	
	\sum_{m\in\Zz}\frac{\,e^{2i\pi mu}\,\,\widetilde\omega(m)}{2i\pi 
		(a-m)} 
		=\, 
		\frac{1}{2i\sin \pi a}\,\int\nolimits_{\Rr}e^{2i\pi a\, ((x))}
		\omega(x-u) dx\,.
\end{align*}
Hence,
\begin{align*}	
\sup_{u\in\Rr}\,\big|\sum_{m\in\Zz}\frac{\widetilde\omega(m)\,e^{2i\pi
 mu}}{2i\pi (a-m)}\,\big| 
	\le\, 
	\frac{1}{2|\sin \pi a|}\,\int\nolimits_{\Rr}\omega(x-u) dx = \frac{1}{2|\sin 
		\pi a|}\, ,
\end{align*}
since $\omega(x)\ge0$.
Applying this inequality to \eqref{eq4.06*}, we obtain the  bound  \eqref{eq4.012}.
	\end{proof}

We now compare $\BBB_d(t;\bfw,\{\bfu\})$ 
 with its convolution \eqref{eq4.09l}.
\begin{lemma}
Let $\bfu\in\UUU(\delta),\,\, |\,\bfu_1-\bfu\, |_{\infty}\le T^{-1},\, |t_1 
-t|\le T^{-1}$, and $T\ge 2\delta^{-1}$. 
Then
\begin{equation}
	\left.
	\begin{aligned}
	\BBB_d(t_1;\bfw,\{\bfu_1\})&=\BBB_d(t;\bfw,\{\bfu\}) + O_{\delta}(T^{-1} 
	t^{d-1})\, ,
	\\
	\BBB_d*\omega_T(t;\bfw,\bfu)&=\BBB_d(t;\bfw,\{\bfu\}) + 
	O_{\delta}(T^{-1}t^{d-1})\,, 
	\\
	\BBB_d*\omega_T(t_1;\bfw,\bfu_1)&=\BBB_d(t;\bfw,\{\bfu\}) + O_{\delta}(T^{-1} 
	t^{d-1})\, .
\end{aligned}
\quad\right\}
\label{eq4.0002****}
\end{equation}
	\end{lemma}
	\begin{proof}
		We have
	\begin{align*}
		\BBB_d(t_1;\bfw,\{\bfu_1\})-&\BBB_d(t;\bfw,\{\bfu\})
		\\
		\notag 
		=& \BBB_d(t_1;\bfw,\{\bfu_1\})-\BBB_d(t;\bfw,\{\bfu_1\}) 
		+\BBB_d(t;\bfw,\{\bfu_1\})-\BBB_d(t;\bfw,\{\bfu\})\, .
	\end{align*}	
$\BBB_d(t;\bfw,\{\bfu\})$ is a polynomial of degree $d$ whose leading term 
independent of $\bfu$. Therefore,  
$\BBB_d(t_1;\bfw,\{\bfu_1\})-\BBB_d(t;\bfw,\{\bfu_1\})=O(T^{-1} t^{d-1})$,
and
	\begin{equation*}
		\BBB_d(t;\bfw,\{\bfu_1\})-\BBB_d(t;\bfw,\{\bfu\}) = \sum_{j=0}^{d-1} 
		\,\big(c_j(\bfu_1)-c_j(\bfu)\big)\,t^j =O_{\delta}(T^{-1} t^{d-1})\, ,
		\label{eq4.0002***}
	\end{equation*}
	since $\bfu_1\in\UUU(\delta/2)$ and $c_j(\bfu_1)$ for $0\le j\le d-1$ are 
	$\CCC^{\infty}$ functions on 
	$\UUU(\delta/2)$. This proves the first bound in \eqref{eq4.0002****}.
		Finally, we have 
		\begin{align}
			\BBB_d*\omega_T(t;\bfw,\bfu)&-\BBB_d(t;\bfw,\{\bfu\})\notag
			\\
			&=
	\int\nolimits_{\Rr}\left(\BBB_d(t;\bfw,\{\bfx\})-\BBB_d(t;\bfw,\{\bfu\})\right)
			\omega_T (\bfu -\bfx) d\bfx\, .
			\label{eq4.a2}
		\end{align}
If $\bfx$ belongs to the support of $\omega_T(\bfx-\bfu)$, then 
$\bfx\in\UUU(\delta/2)$. Combining \eqref{eq4.a2} with the first bound in 
\eqref{eq4.0002****}, we obtain the second and third bounds in \eqref{eq4.0002****}.
		\end{proof}

Set $t^{\pm}=t\pm 2T^{-1},\,\bfv^{\pm}=\mp T^{-1}\bfe_1$;
we assume that $t$ and $T$ are large enough so that $t^{-}>0$.
Consider the simplices 
$\Delta^{-}(t^{\pm};\bfw,\bfv^{\pm})=\Delta^{-}(t^{\pm};\bfw)+\bfv^{\pm}$:
\begin{align*}
	\Delta^{-}(t^{+};\bfw, \bfv^{+})
	&=\Delta^{-}(t+2T^{-1}\bfe_1\centerdot\bfw;\,\bfw,\, -T^{-1}\bfe_1)\notag
	\\
	&=\{\bfx: x_j>-T^{-1}, j\in [d],\,\, 
	\bfx\centerdot\bfw\,<t+T^{-1}\bfe_1\centerdot\bfw \}\, ,\notag
	\\
	\Delta^{-}(t^{-};\bfw,\bfv^{-})
	&=\Delta^{-}(t-2T^{-1}\bfe_1\centerdot\bfw;\,\bfw, T^{-1}\bfe_1)\notag
	\\
	&=\{\bfx: x_j>T^{-1}, j\in [d],\,\, 
	\bfx\centerdot\bfw\,<t-T^{-1}\bfe_1\centerdot\bfw \}\, ,
\end{align*}
Clearly, 
$\Delta^{-}(t^{-};\bfw,\bfv^{-})\subset\Delta^{-}(t;\bfw)
\subset\Delta^{-}(t^{+};\bfw, 
\bfv^{+})\, .$
Define the convolution:
\begin{equation*}
	\chi*\omega_T(t;\bfw,\bfu)	
	=\int\nolimits_{\Rr^d}\chi
	(\Delta^{-}(t;\bfw),\bfx)
	\,\,\omega_T(\bfu-\bfx)\, d\bfx\, .
\end{equation*}

\begin{lemma}
For $\bfu\in\Rr^d$, we have the inequalities
\begin{equation}
	\chi*\omega_T(t^{-};\bfw,\bfu-\bfv^{-})\le \chi(\Delta^{-}(t;\bfw),\bfu) 
	\le\chi*\omega_T(t^{+};\bfw,\bfu-\bfv^{+})
	\label{eq4.0015}
	\end{equation} 
	and 
	\begin{equation}
N*\omega_T(t^{-};\bfw,\bfu+\bfv^{-})\le N(t;\bfw,\bfu)\le 
N*\omega_T(t^{+};\bfw,\bfu+\bfv^{+})\, .		
		\label{eq4.0016}
	\end{equation} 
\end{lemma}
\begin{proof}
$\omega_T(\bfu-\bfx)$
is supported on $[-1/(2T),1/(2T)]^d+\bfu$, and $0\le\chi*\omega_T(t;\bfu)\le1$. 
The simplices $\Delta^{-}(t^{+};\bfw, \bfv^{+})$ and $\Delta^{-}(t^{-};\bfw, 
\bfv^{-})$ 
are chosen so that the following holds:

\textit{(i)} If $\bfu\notin\Delta^{-}(t^{-};\bfw,\bfv^{-})$, then 
$\chi*\omega_T(t^{-};\bfw,\bfu-\bfv^{-})=0$,  and the 
left inequality in \eqref{eq4.0015} follows. 

\textit{(ii)} If 
$\bfu\in\Delta^{-}(t^{+};\bfw,\bfv^{+})$, 
then 
$\chi*\omega_T(t^{+};\bfw,\bfu-\bfv^{+})=1$, and the right inequality in
\eqref{eq4.0015} 
follows. 

It follows from \eqref{eq3.1} and \eqref{eq4.09l} that
\begin{equation}
N*\,\omega_T(t;\bfw,\bfu)=\sum\nolimits_{\bfm\in\Zz^d}\,
\chi*\,\omega_T(t;\bfw,\bfm-\bfu)\, .
\label{eq4.0017}	
	\end{equation}
Indeed,
\begin{align*}
	N*\,\omega_T(t;\bfw,\bfu)&=\sum\nolimits_{\bfm\in\Zz^d}
\int\nolimits_{\Rr^d}\chi(\Delta^{-}(t;\bfw),\bfm-\bfx)\,\omega_T(\bfu-\bfx)\, 
d\bfx
	\notag
	\\	
	&=\sum\nolimits_{\bfm\in\Zz^d}
	\int\nolimits_{\Rr^d}\chi(\Delta^{-}(t;\bfw),\bfy)\,
	\omega_T(\bfm -\bfu-\bfy)\, d\bfy
	\notag
	\\
	&=\sum\nolimits_{\bfm\in\Zz^d}\,\chi*\,\omega_T(t;\bfw,\bfm-\bfu)
	\, ,
	\end{align*}
since $\omega_T(x)$ is an even function.

Replacing $\bfu$ in \eqref{eq4.0015} by $\bfm-\bfu$ and summing  over 
$\bfm\in\Zz^d$, we obtain \eqref{eq4.0016} in view of
 \eqref{eq3.1} and \eqref{eq4.0017}. 
\end{proof}                                                                     

Combining Proposition 7.1 with Lemmas 7.2 and 7.3, we 
obtain the 
following.                                                                          
\begin{proposition}
Let inclines $\Theta_{\widehat\bfw}$ be $\kappa$-multiplicatively 
approximable, $\bfu\in\UUU(\delta)$, and $T\ge 2\delta^{-1}$.
Then the error term  \eqref{eq3.28**cc} satisfies the bound
\begin{equation}
	\RRR_{\delta}(t;\bfw)
	\,\le\,
\,\frac{1}{2^{d-1}\pi}\,\,	\,\SSS_0(\,\Theta_{\widehat\bfw},\, \omega_T)\,
+O(T^{-1}\,t^{d-1})\,,	
	\label{eq4.0018}
\end{equation}
where $\SSS_0(\,\Theta, \,\omega_T)$ is defined in Proposition~7.1.
\end{proposition}
\begin{proof}
The inequalities \eqref{eq4.0016} can be written as
\begin{align*}
	&\frac{1}{d!\,w_1\dots w_d} \,\,
	\BBB_d*\omega_T\,(t^{-};\bfw,\bfu+\bfv^{-})
	+R*\omega_T\,(t^{-};\bfw,\bfu+\bfv^{-})
	\\
	&\le\frac{1}{d!\,w_1\dots w_d} \,\,
	\BBB_d\,(t;\bfw,\bfu)+R\,(t;\bfw,\bfu)
	\\
	&\le\frac{1}{d!\,w_1\dots w_d} \,\,
	\BBB_d*\omega_T\,(t^{+};\bfw,\bfu+\bfv^{+})
	+R*\omega_T\,(t^{+};\bfw,\bfu+\bfv^{+})
	\, .	
\end{align*}
Lemma~7.2 implies 
$	\BBB_d*\omega_T(t^{\pm};\bfw,\bfu+\bfv^{\pm})=\BBB_d(t;\bfw,\{\bfu\}) + 
O_{\delta}(T^{-1} 
	t^{d-1})\, .$
Hence,
\begin{align*}
	|R(t;\bfw,\bfu)|\,\le\,
\max\{|R*\omega(t^{-};\bfw,\bfu+\bfv^{-})|,\,|R*\omega(t^{+};\bfw,\bfu+\bfv^{+})|\}
	+O(T^{-1}t^{d-1})\, .
\end{align*}
This inequality, together with  Proposition~7.1, implies \eqref{eq4.0018}.
\end{proof}

The Main Theorem now follows directly from Proposition~7.2.
\begin{proof}[\bf{Proof of Theorem~4.1}]
For a given $\epsilon>0$, write the sum $\SSS_0(\,\Theta, \,\omega_T)$ as 
follows
$$ \SSS_0(\,\Theta_{\widehat\bfw}, 
\,\omega_T)=\sum_{j=1}^d\,\sum\nolimits_{1\le m\le 
T^{1+\epsilon}} + \sum_{j=1}^d\,\sum\nolimits_{m> 
T^{1+\epsilon}}\,.$$
For the first sum, we have
\begin{align*}
\sum_{j=1}^d\,\sum\nolimits_{1\le m\le T^{1+\epsilon}}& = \,
\sum_{j=1}^d\,\sum\nolimits_{m\le T^{1+\epsilon}}\,
\frac{|\widetilde\omega(T^{-1}\,m)|}{m\prod\nolimits_{l\in [d], l\ne 
j}|\sin(\pi\,\theta_{j,l}m)|}
\notag
\\
&\le\sum_{j=1}^d\,\sum\nolimits_{m\le T^{1+\epsilon}}\,
\frac{1}{m\prod\nolimits_{l\in [d], l\ne j}|\sin(\pi\,\theta_{j,l}m)|}
\notag
\\
&=S_0(\,\Theta_{\widehat\bfw}, \,T^{1+\epsilon})\approx 
S(\,\Theta_{\widehat\bfw}, 
\,T^{1+\epsilon})\,,
\end{align*}
where $S_0(\,\Theta_{\widehat\bfw}, \,T)$ and $S(\,\Theta_{\widehat\bfw}, \,T)$ 
are defined in \eqref{eq1D}.

Since $\widetilde\omega(x)\in\frak{S}$, for the second sum we have
\begin{align*}
	\sum_{j=1}^d\,&\sum\nolimits_{m> T^{1+\epsilon}}= \,
	\sum_{j=1}^d\,\sum\nolimits_{m> T^{1+\epsilon}}\,
	\frac{|\widetilde\omega(T^{-1}\,m)|}{m\prod\nolimits_{l\in [d], l\ne 
	j}|\sin(\pi\,\theta_{j,l}m)|}
	\notag
	\\
	&\ll \sum\nolimits_{m> T^{1+\epsilon}}\,
	m^{\kappa}\,(1+T^{-1} m)^{-A}
	\ll T^A\, \sum\nolimits_{m> T^{1+\epsilon}}\,m^{\kappa -A}
	\notag
	\\
	&\ll T^{(\kappa +1)(1+\epsilon)-\epsilon A}
	= O_a(T^{-a})
\end{align*}
with an arbitrarily large $a>0$; here we used \eqref{eq4.14} and the fact that 
$\widetilde\omega(x)\in\frak{S}$. Therefore, 
$$\SSS_0(\,\Theta_{\widehat\bfw}, \,\omega_T)\ll S(\,\Theta_{\widehat\bfw}, 
\,T^{1+\epsilon})\,.$$ 
Substituting this bound into \eqref{eq4.0018}, we obtain \eqref{eq1.31aa}.
The proof is complete.
\end{proof}

\end{document}